\documentclass[11pt]{amsart}
\usepackage{amsaddr}
\usepackage[letterpaper, hmargin=1.0in, vmargin=1.0in]{geometry}
\usepackage[bottom]{footmisc} 
\usepackage{import}
\usepackage{xifthen}
\usepackage{pdfpages}
\usepackage{transparent}
\usepackage{cancel}

\usepackage{tikz}
\usepackage{float}
\usepackage{mathtools, amsmath, amsfonts, amssymb, amsthm}
\usepackage{pgffor}
\usepackage{etoolbox}
\usepackage[shortlabels]{enumitem}
\usepackage[unicode=true, colorlinks=true, linkcolor=blue]{hyperref}
\hypersetup{pdfauthor={Name}}
\usepackage{setspace}

 \usepackage{xcolor}

 \newcommand{\eps}{\varepsilon}

 \newcommand{\Ber}{\mathrm{Ber}}

 \newcommand{\Om}{\Omega}
 
 \renewcommand{\O}[1]{\mathcal{O}\Rnd{#1}}

 \newcommand{\D}{\Delta}

 \newcommand{\al}{\alpha}
 \newcommand{\be}{\beta}

 \newcommand{\Exp}[1]{\exp\left(#1\right)}
 
 \newcommand{\Frac}[2]{\left(\dfrac{#1}{#2}\right)}
 \renewcommand{\Box}[1]{\left[#1\right]}
 \newcommand{\Rnd}[1]{\left(#1\right)}

 \makeatletter
 \newcommand{\E}[1][\@nil]{%
 	\def\tmp{#1}%
 	\ifx\tmp\@nnil
 	\mathbb{E}
 	\else
 	\mathbb{E}\left[#1\right]
 	\fi}
 
 \renewcommand{\P}[1][\@nil]{%
 	\def\tmp{#1}%
 	\ifx\tmp\@nnil
 	\mathbb{P}
 	\else
 	\mathbb{P}\Rnd{#1}
 	\fi}
 \makeatother

 \newcommand{\mnorm}[1]{{\left\vert\kern-0.25ex\left\vert\kern-0.25ex\left\vert #1 
 		\right\vert\kern-0.25ex\right\vert\kern-0.25ex\right\vert}}

 \newtheorem{thm}{Theorem}
 \newtheorem{lem}[thm]{Lemma}
 \newtheorem{defn}{Definition}
 \theoremstyle{remark}
 
 \newcommand{\1}{\mathbf{1}}
 \newcommand{\floor}[1]{\lfloor #1 \rfloor}

 \newcommand{\cG}{\mathcal{G}}

 \newcommand{\Bin}{\mathrm{Bin}}

 \newcommand{\f}[2]{\frac{#1}{#2}}

 \newcommand{\df}{\stackrel{\text{def}}{=}}

 \newcommand{\rlem}[1]{Lemma \ref{#1}}

 \newcommand{\ceil}[1]{{\lceil #1 \rceil}}
 \renewcommand{\floor}[1]{{\lfloor #1 \rfloor}}
 \newcommand{\poly}{\mathrm{poly}}
 \newcommand{\nul}{\varnothing}

\usepackage[style=alphabetic]{biblatex}
\newcommand{\fix}[1]{#1}
\newcommand{\fixr}[1]{}

\newcommand{\ER}{ Erdős–Rényi }
\renewcommand{\o}{\widetilde{o}}
\renewcommand{\O}{\widetilde{O}}
\renewcommand{\Om}{\widetilde{\Omega}}
\def\G{{G^*}}

\newcommand{\tE}{\widetilde{E}}

\newcommand{\Q}{Q}

\newcommand{\pln}{\poly\log n}

\begin{document}
\title[Upper tails in Poisson regime]{Universality in prelimiting tail behavior for regular subgraph counts in the Poisson regime}
\author{Mriganka Basu Roy Chowdhury}
\date{}

\begin{abstract}
    \footnotesize
    Let $N$ be the number of copies of a \fix{small subgraph} $H$ in an\ER graph $G \sim \cG(n, p_n)$ where $p_n \to 0$ is chosen so that \fix{$\E N = c$, a constant}. \fix{Results of Bollobás \cite{bollobas} show} that for regular graphs $H$, the count $N$ weakly converges to a Poisson random variable. For large but finite $n$, and for the specific case of the triangle, investigations of the upper tail $\P(N \geq k_n)$ by \cite{ganguly} revealed that there is a phase transition in the \fix{tail behavior and the associated mechanism}. Smaller values of $k_n$ correspond to disjoint occurrences of $H$, leading to Poisson tails, with a different behavior emerging when $k_n$ is large, guided by the appearance of an almost clique. We show that a similar phase transition  also occurs when $H$ is any regular graph, at the point where $k_n^{1 -2/q}\log k_n = \log n$ ($q$ is the number of vertices in $H$). This establishes universality of this transition, previously known only for the case of the triangle.
\end{abstract}

\maketitle{}

\tableofcontents
\singlespacing

\normalsize
\section{Introduction}
Let $\cG(n, p)$ be the\ER random graph where each edge is present independently with probability $p$. Over the last few decades, various properties of these graphs have been investigated, resulting in a very fine understanding of this model. For a fixed\fixr{small} graph $H$, let $\Q_H(G)$ be the number of copies of $H$ in $G \sim \cG(n, p)$. Starting with the seminal paper of Erdős and Rényi \cite{erdosrenyi}, the distribution of $\Q_H(G)$, both for fixed $p$ and for $p$ varying with $n$ at different rates, has been the subject of intense research. For example, the probability of $\Q_H(G) > 0$ was studied by \cite{schurger1979limit} where $H$ was a complete graph. This work was later extended to other subgraphs by Bollobás \cite{bollobas}. In addition to this,  in many cases we also know asymptotic distributions of $\Q_H(G)$. A criterion for asymptotic normality of $\Q_H(G)$ was obtained by Ruciński \cite{rucinski1988small}. This however must exclude the regime of $p_n$ where we expect $\Theta(1)$ copies of the subgraph in expectation where asymptotic normality cannot hold. A natural guess for the limiting distribution in this case is Poisson, but it turns out that it is not true in general. However, for the large class of \emph{strictly balanced graphs} (which includes all regular graphs), asymptotic Poisson behavior was established by Bollobás \cite{bollobas} (see the discussion at the beginning of section \ref{sec:lowerbounds} for the definition of strictly balanced graphs).

A particularly active direction of research has been the study of the upper tail probability $\P(\Q_H(G) \geq k_n)$, the so-called ``infamous upper tail''. The seminal paper of Chatterjee and Varadhan \cite{chavar} achieved asymptotically sharp estimates on this for the case of fixed $p_n = p > 0$, using the theory of dense graph limits and graphons (see \cite{lovasz:book} or \cite{chatterjee:book}) as well as the celebrated Szemerédi regularity lemma \cite{szemeredi}, by reducing the computation of the upper tail probability to a natural ``mean-field'' variational problem. However due to poor quantitative estimates in the regularity lemma, arguments in \cite{chavar} could only be extended to $p_n$ decaying as a negative power of $\log n$, see \cite{lubetzkyzhaovariational}. The most widely studied case is of $H = K_3$, the triangle. Here the arguments of Chatterjee and Dembo \cite{chadem} and Lubetzky and Zhao \cite{lubetzkyzhaovariational} extended previously known results down to $p_n \geq n^{-1/42}\log n$. However the mean-field variational formulation is expected to hold as long as $p_n \gg \log n / n$. After a series of improvements  (Augeri \cite{augeri2018nonlinear}, Cook and Dembo \cite{cook2020large}, Eldan \cite{eldan2018gaussian}),   a breakthrough by Harel, Mousset and Samotij \cite{samotij} essentially solved the upper tail problem in the regime $p_n \gg 1/n$, using methods inspired by the classical moment arguments of Janson, Oleszkiewicz and Ruciński \cite{janson2004upper}, leaving the interesting case of $p_n = \Theta(1/n)$ still open. 
Sharp understanding of the upper tail (for the case of $K_3$) in this sparse regime was finally achieved by Ganguly, Hiesmayr and Nam \cite{ganguly} and Chakraborty, \fix{van der Hofstad} and den Hollander \cite{suman}, thereby completing our understanding of the upper tail in the case of the triangle. 

\fix{Our interest is primarily in an analogous regime for general regular graphs $H$, where $p_n \to 0$ in such a way that $\E \Q_H(G) = \Theta(1)$ as $n \to \infty$}. To see what this rate of decay of $p_n$ might be, let $H$ be a $q$-vertex $\D$-regular graph. Observe that the expected number of copies of $H$ in $G$ is 
\begin{align*}
    \binom{n}{q} p_n^{-q\D/2}
\end{align*}
up to constants depending only on $H$ \fixr{It therefore follows from elementary inequalities}. Since $(n/q)^q \leq \binom{n}{q} \leq (ne/q)^q$, the choice $p_n = \Theta(n^{-2/\D})$ ensures that this quantity is $\Theta(1)$. As discussed earlier, the results of \cite{bollobas} show that $\Q_H(G)$ is asymptotically Poisson in this case (hence the term ``Poisson regime''). However, the arguments there do not yield effective quantitative prelimiting estimates. In the case of $H = K_3$, the triangle, the prelimiting upper tail was investigated in \cite{ganguly}, where it was shown that the behavior of $\P(\Q_H(G) \geq k_n)$ \fix{undergoes a transition when $k_n$ is around $O(\log^3 n)$ (that the behavior after this transition persists till the maximum of $k_n = \Theta(n^3)$ was proved in \cite{suman}). The transition is as follows: when $k_n$ is smaller than the threshold, the dominant mechanism is the occurrence of many disjoint copies, but when $k_n$ is above the threshold, the mechanism changes to having an (almost) clique. }

In this paper we show that this phase transition also happens when $H$ is any regular graph. As above, we will fix $H$ to be a connected $q$-vertex $\D$-regular graph (with $q \geq 3$) throughout this paper. Our main result is the following.
\begin{thm}
    \label{thm:main}
    Let $k_n \geq 2$ and $p_n = n^{-2/\D}$. If $G \sim \cG(n, p_n)$, the probability of $Q_H(G) \geq k_n$ satisfies
    \begin{align*}
        C_1\Exp{-C_2 L_n} & \leq \P(\Q_H(G) \geq k_n) \leq C_3\Exp{-C_4 L_n}, \quad \text{where,} \\
        L_n & \df \min(k_n \log k_n, k_n^{2/q}\log n).
    \end{align*}
    Here the constants $C_1, C_2, C_3, C_4$ that only depend on $H$. (This choice of $p_n = n^{-2/\D}$ ensures that $\E \Q_H(G) = \Theta({1})$.) 
\end{thm}

It is easy to extend this result for $p_n = c' n^{-2/\D}$ for some constant $c'$, but we make this choice for notational simplicity. Theorem \ref{thm:main} almost achieves the sharp phase transition at $k_n$ satisfying $k_n^{1 - 2/q} \log k_n = \log n$, as obtained in \cite{ganguly} for the special case of $H = K_3$.

{\bf Note:} Throughout this paper, a \emph{copy} of a graph $H$ in $G$ is counted only up to automorphisms. Equivalently, it is a subset of edges of $G$ forming a graph isomorphic to $H$. Of course, even if we did count automorphisms separately, the form of our result would not change due to the presence of $H$ dependent constants.

\vskip1em
\noindent\textbf{Acknowledgements: } I would like to thank my advisor Prof. Shirshendu Ganguly for suggesting the problem and for numerous discussions and helpful insights. I also had several interesting and informative discussions with Ella Hiesmayr during the initial stages of this project.
\vskip2em
\renewcommand{\H}{\textsf{H}}
\section{Idea of proof}
\label{idea}
\fix{We begin by setting up the notations. For any graph $G$, we will use  $V(G), E(G)$ to denote the vertex set and the edge set respectively, while $v(G), e(G)$ will be the respective sizes of these sets. We will write $d_v$ to denote the degree of a vertex $v$. $C = C(H)$ will always denote a constant that depends only on $H$.}

There are two main subparts to the proof. The first part deals with the case when $k_n$ is small, and the second part is for the case of large $k_n$. It turns out that arguments in the first part works for all $k_n \leq n^c$ for some small $c$ that depends on $H$, and the second part works for all $k_n$ that grow faster than everything that is $\poly \log n$. It is then clear that these two proofs cover all possible values of $k_n$. The phase transition will occur in the former case, while the purpose of the latter is to show that the behavior shown in the first part persists after the phase transition all the way up to the maximum possible value  of $k_n$.

The proof for the case of small $k_n$ is via a reduction to the analysis of the structure of \fix{``spanned graphs''}. This is inspired by \cite{ganguly} which treats the same problem in the case of the triangle. However, informally, that paper uses the fact that since the critical probability for a triangle is $p_n = O(1/n)$, the probability that a given subgraph with $v$ vertices and $e$ edges occurs as an isomorphic copy is at most
\begin{align*}
    \binom{n}{v} p_n^{-e} \leq n^{-(e - v)}.
\end{align*}
Thus this probability is directly related to the quantity $e - v + 1$ which can be interpreted as the number of ``excess edges'' in the subgraph after a spanning  tree has been chosen. Such an idea fails in our case because for a general $\D$-regular graph the critical probability is  instead
$p_n = n^{-2/\D}$, so that the above bound is now $n^{-(2e/\D - v)}$, for which a similar interpretation is not available. 
To circumvent this issue,
we prove that even in this general case, any ``spanned graph'' $G'$ must satisfy $2e(G')/\D - v \geq Ce(G')$ for some constant $C$, see \rlem{lem:spannedstruct}. Armed with this inequality, we bound the probability that the graph contains a certain spanned graph $G'$ which has many copies of $H$. Due to the last inequality, this is now related to the minimal number of edges in a candidate $G'$ containing, say, $\ell$ copies of $H$, \fix{which is at least $\ell^{2/q}$ up to constants (due to \rlem{lem:holder}).} It turns out that the concavity of this map $\ell \mapsto \ell^{2/q}$ implies that
it is either optimal to have all the copies be in a \fix{large spanned component,} or occur disjointly, depending on how large $k_n$ is. As a consequence of this dichotomy, we obtain our phase transition. \fix{However, a formal proof requires one to also control the entropy stemming from the number of possible ways to distribute $k_n$ copies of $H$ among spanned graphs.} This is \fix{done using a dyadic} decomposition argument where we classify spanned components into logarithmic many collections depending on the number of copies they contain. See \rlem{lem:few} for more details.

In the case of large $k_n$, the main idea is that if a graph has too many copies of $H$, then there must occur certain subgraphs, such that the expected count of $H$ conditioned on their occurrence in the random graph is near optimal (see \rlem{lem:seed} for more details). \fix{This idea first appeared in \cite{samotij} (who called such subgraphs ``cores''),} and was adopted by \cite{suman} to prove a result similar to ours but only for the triangle. However a n\"aive implementation of the approaches found in these papers fail in our case. As \rlem{lem:seed} will show, the upper tail bound essentially boils down to upper bounding \emph{the occurrence probability of a core}. This involves an efficient union bound over all cores. To illustrate the difficulties involved, let us see why a n\"aive direct count approach already fails for the case of $H = C_4$, the cycle on 4 vertices. In the  case $k_n \sim n$, it turns out that all the complete bipartite graphs $K_{2, m}$ 
    (i.e. 2 vertices on one side) with $m = O(\sqrt{n})$ qualify as cores.
Note that their count is of the order of $\binom{n}{m} = \exp(O(\sqrt{n} \log n))$ whose exponent exactly matches the probability bound of $C'\exp(-C\sqrt{n}\log n)$ as claimed by Theorem \ref{thm:main} in this case. This reduces the direct union bound approach to a delicate game of constants.

\fixr{It is worth noting that the ideas of \cite{samotij} extend to the completely sparse regime, as was demonstrated by \cite{suman}.}

Our proof avoids these issues, and instead crucially uses two estimates:
    \begin{enumerate}
        \item \fix{an upper bound on the number of copies of $H$ in another graph $G$ in terms of $e(G)$}, 
        \item and, a lower bound on the product of degrees $d_u d_v$ for  all edges $(u, v)$ in a core. 
    \end{enumerate}
    It turns out that for simple graphs $H$, these estimates can follow from direct combinatorial arguments that exploit the specific structure of $H$. For example, in the case of cycles, a standard spectral argument is sufficient to deduce an (essentially optimal) upper bound on the number of copies of it in a graph with $e$ edges. Further, the inequality involving the product of degrees can also be established via direct combinatorial methods.  We expect other simple classes of graphs to also yield to similar arguments. Note that this is similar in spirit to \cite{suman} whose analysis also crucially relied on the geometry of the triangle. But since we wish to deal with all regular graphs at once, we must employ alternate means to deduce information about the  structure of these cores. 
    \fix{For this reason, we develop a novel strategy invoking \emph{Finner's inequality} \cite{finner}. This is a form of H\"older's inequality which has proved to be particularly useful in the context of graphons, as demonstrated, for example, in \cite{lubetzky2015replica}} (see section \ref{sec:holder}). A straightforward application of this inequality yields the desired bound on the number of copies of $H$ in a graph with $e$ edges, as proved in \rlem{lem:holder}. But this alone is not sufficient for the proof. A much more involved analysis also allows us to also prove bounds on the expected count of $H$ in a graph containing a specific edge (\rlem{lem:fixedholder}). These results allow us to not 
only show that cores themselves contain many copies of $H$ (\rlem{lem:manycopies}), but also that each edge in the core has the interesting 
property of having a large product of degrees (\rlem{lem:product}), as asserted in point (2) above. All these facts are finally combined in the dyadic decomposition argument of \rlem{lem:many},
which essentially exploits the fact that since the product of degrees is high, very low degree vertices can only have edges to vertices with very high degree. Since there are few vertices with very high degree, this allows us to cheaply account for the neighbor sets of these low degree vertices, which are
the main sources of trouble in the proof. 

We end this section with a word about asymptotic notation. In addition to the usual $O(\cdot), \Omega(\cdot), \Theta(\cdot)$, we also use $\O(g(n)) \ni f(n)$ to mean that $f(n) / \pln = O(g(n))$, and $\Om(g(n)) \ni f(n)$ to mean that $f(n)\pln = \Om(g(n))$. Finally, $\o(g(n)) \ni f(n)$ if and only if $f(n)h(n) = o(g(n))$ for every $h(n)$ that is $\pln$. So, for example, $n^{0.9} = \o(n)$ but $n/\log^{100} n \neq \o(n)$. In fact $n/\log^{100} n = \Om(n)$.

\section{Lower bounds}
\label{sec:lowerbounds}

This short section collects the necessary results allowing us to assert the lower bound in our main theorem. 

Define the \emph{density} of a graph $H$ to be the ratio $d(H) \df e(H)/v(H)$ (see \cite[Chapter 3]{janson2011random:book}), and call a connected graph \emph{strictly balanced} if for any proper subgraph $H' \subsetneq H$, $d(H') < d(H)$. It is easy to see that $\D$-regular graphs are strictly balanced, because the entire graph has density $\D/2$, but any proper subgraph must have some vertex with degree $< \D$, so that the density is strictly smaller than $\D/2$. By the Poisson limit theorem for subgraph counts at the threshold\footnote{This terminology is often used in the literature to mean a choice of $p_n$ for which $\E \Q_H(\cG(n, p_n)) = \Theta(1)$. } due to Bollobás \cite{bollobas}, the asymptotic distribution of $\Q_H(\cG(n, p_n))$ is therefore Poisson. This fact is also used in the proof of the lemma below, which asserts a lower bound on the probability of having disjoint copies of $H$.

\begin{lem}[Disjoint occurrence]
    Let $D_s$ denote the event that there are $s$ disjoint copies of $H$ in $G \sim \cG(n, p = p_n)$. Fix $\eps > 0$. If $s \leq n^{1 - \eps}$,
    \begin{align*}
        \P(D_s) \geq C\Exp{-C's \log s}
    \end{align*}
    for all sufficiently large $n$ depending on $\eps$. Here $C, C'$ are constants depending on $H$.
\end{lem}
    The matching upper bound is stated in \rlem{lem:disjointprob}.
\begin{proof}
    Divide the vertex set $[n]$ of $K_n$ into $s$ (almost) equal sized groups of size either $\floor{n/s}$ or $\ceil{n/s}$. Set $m \df \floor{n/s}$. Then,
    \begin{align}
        \label{lb:initial}
        \P(s \text{ disjoint copies of } H) \geq \P(\cG(m, p_n) \text{ has a copy of } H)^s
    \end{align}
    To lower bound the probability of the latter, let $q = m^{-2/\D}$, and observe that $p_n = n^{-2/\D} \geq (2ms)^{-2/\D}$, as long as 
    \begin{align*}
        2ms \geq n \iff 2\floor{n/s} \geq n/s,
    \end{align*}
    which happens for all sufficiently large $n$ depending on $s$ (and uniformly in $s \leq n^{1 - \eps}$ depending on $\eps$). We assume this holds below. Then
    \begin{align}
        \label{lb:other}
        \P(\cG(m, p_n) \text{ has a copy of } H) \geq \P(\cG(m, (2ms)^{-2/\D}) \text{ has a copy of } H).
    \end{align}
    We now employ a two stage sampling procedure as follows: to sample a $\cG(m, (2ms)^{-2/\D})$, we first sample a $\cG(m, m^{-2/\D})$, and then keep each of its edges independently with probability $(2s)^{-2/\D}$. It is clear that the resulting graph is a sample from $\cG(m, (2ms)^{-2/\D})$. But crucially, since the expected number of copies of $H$ in $\cG(m, m^{-2/\D})$ is $\Theta(1)$, we can employ the classical fact that the distribution of subgraph counts at the threshold is Poisson for strictly balanced $H$ (see, for instance, \cite{bollobas} or \cite[Theorem 3.19]{janson2011random:book}), to conclude that 
    \begin{align*}
        \P(\cG(m, m^{-2/\D}) \text{ has a copy of } H) \geq c,
    \end{align*}
    for a constant $c$ depending on $H$ for all sufficiently large $m$. A given copy of $H$ in $\cG(m, m^{-2/\D})$ survives the second round of the procedure only if all its edges survive, which happens with probability 
    \begin{align*}
        \Rnd{(2s)^{-2/\D}}^{e(H)} = (2s)^{-\f{2}{\D} \f{q\D}{2}} = (2s)^{-q}.
    \end{align*}
    Putting these together, 
    \begin{align*}
        \P(\cG(m, (2ms)^{-2/\D}) \text{ has a copy of } H) \geq c(2s)^{-q},
    \end{align*}
    so that we can invoke \eqref{lb:initial} and \eqref{lb:other} to obtain
    \begin{align*}
        \P(s \text{ disjoint copies of } H) \geq \Rnd{c(2s)^{-q}}^s \geq C\Exp{-C' s \log s},
    \end{align*}
    for sufficiently large $n$ depending on $\eps$.
\end{proof}

The next lemma is a lower bound on clique occurrence probabilities.

\begin{lem}[Clique]
    The probability that there is a clique of size $s \geq 2$ in $G \sim \cG(n, p = p_n)$ is at least $\Exp{-c s^2 \log n}$ where $c$ depends only on $\D$.
\end{lem}
\begin{proof}
    The probability of having a clique of size $s$ is at least that of the event that any fixed set of  $s$ vertices have all edges in between themselves. The probability of this latter event is exactly:
    \begin{align*}
        p_n^{\binom{s}{2}} \geq \Exp{-\f{1}{2}(s - 1)^2 \log (1/p_n)} = \Exp{-c s^2 \log n},
    \end{align*}
    where $c = c(\D)$. This proves the result.
\end{proof}

These two results together imply
\begin{lem}[Lower bound of Theorem \ref{thm:main}]
    In the notation above, 
    \begin{align*}
        \P(\Q_H(G) \geq k_n) \geq C\Exp{-C'L_n} = C\Exp{-C'\min(k_n \log k_n, k_n^{2/q} \log n)}
    \end{align*}
    for constants $C, C'$ depending on $H$.
\end{lem}
\begin{proof}
    Given the lemmas above, the only remaining part is that the smallest clique containing at least $k_n$ copies of $H$ is of size $Ck_n^{2/q}$, but this is immediate from a counting argument.
\end{proof}

\section{Finner's inequality}
{
\label{sec:holder}
As described earlier, one of the most important results we require throughout the sequel is the following variant of H\"older's inequality. As far as we could tell, applications of this inequality in the context of large deviations on random graphs first appeared in \cite{lubetzky2015replica}, where they used it to analyze the variational problem associated with the upper tail deviations for $\cG(n, p)$ with fixed $p$.
    \renewcommand{\Om}{\Omega}
    \begin{lem}[Finner's inequality, \cite{finner}]
        \label{lem:finner}
        Let $\mu_1, \mu_2, \ldots, \mu_n$ be probability measures on $\Om_1, \Om_2, \ldots, \Om_n$ respectively and define $\Om = \prod_{i = 1}^n \Om_i$ and $\mu = \prod_{i = 1}^n \mu_i$. Also let $A_1, A_2, \ldots, A_m$ be nonempty subsets of $[n] = \{1, \ldots, n\}$, and for any subset $A \subseteq [n]$, denote $\mu_A = \prod_{i \in A} \mu_i$ and $\Om_A = \prod_{i \in A} \Om_i$. Suppose $f_i \in L^{p_i}(\Om_{A_i}, \mu_{A_i})$ for each $i \in [m]$, such that for all $x \in [n]$, 
        \begin{align*}
            \sum_{i : x \in A_i} p_i^{-1} \leq 1.
        \end{align*}
        Then we have the inequality
        \begin{align*}
            \int \prod_{i = 1}^m |f_i| d\mu \leq \prod_{i = 1}^m \Rnd{\int |f_i|^{p_i} d\mu_{A_i}}^{1/p_i}.
        \end{align*}
        Here, on the left-hand side, one should think of each $f_i$ extended to a function on $\Om$ which depends only on the coordinates in $A_i$.
    \end{lem}

    Our primary application of this will be to count homomorphisms from $H$ to $G$, using the \emph{graphon representation} of $G$. 

    \begin{defn}[Graphon representation of a finite graph $G$]
        Let $G$ be a finite graph on $n$ vertices, and without loss of generality, suppose its vertices are $v_0, \ldots, v_{n - 1}$. The graphon (representation) of $G$ is defined as a function $f_G : [0, 1]^2 \to \{0, 1\}$ as follows. 
        \begin{align*}
            f_G(x, y) = 1 \iff \exists\ 0 \leq i, j < n : \f{i}{n} \leq x < \f{i + 1}{n}, \f{j}{n} \leq y < \f{j + 1}{n}, (v_i, v_j) \in E(G)
        \end{align*}
        else it is zero.
    \end{defn}
    It is a classical result following from direct computation that the number of homomorphisms from $H$ to $G$ is given by (see, for instance, \cite[Chapter 7]{lovasz:book}, or \cite[Chapter 3]{chatterjee:book}),
    \begin{align*}
        n^q \int_{[0, 1]^q} \prod_{(u, v) \in E(H)} f(t_u, t_v) \prod_{u \in V(H)} dt_u.
    \end{align*}
    This would enable us to bound the number of homomorphisms from $H$ to $G$, and therefore the number of copies of $H$ in $G$.
    
    Later, we will need a slightly more general result relaxing the regularity requirement for $H$. While it may also be deduced from \cite[Corollary 3.2]{lubetzky2015replica}, we provide a proof for completeness.
    \begin{lem}
        \label{lem:holder}
        Let $H$ be $q$-vertex $e$-edge (not necessarily connected) graph with each degree bounded above by $\D$. Let $G$ be a graph with graphon $f_G$. Write $n = |V(G)|$, and $m = |E(G)|$. Then the number of homomorphisms of $H$ in $G$, $N_H(G)$ is at most
        \begin{align*}
            N_H(G) \leq C n^{q - 2e/\D} m^{e/\D}.
        \end{align*}
        In particular for a regular $H$, $2e/\D = q$, so we have $N_H(G) \leq C m^{q/2}$.
    \end{lem}
    \begin{proof}
        We apply Finner's inequality with $\Om_i = [0, 1]$ for all $i \in [n]$, and $\mu_i$ being the uniform measure on $[0, 1]$. For each edge $e = (u, v)$ of $H$ we choose $A_e = \{u, v\}$ and $p_e = \Delta$. Then the conditions are satisfied, and we get
        \begin{align*}
            \int_{[0, 1]^q} \prod_{(u, v) \in E(H)} f_G(t_u, t_v) \prod_{u \in V(H)} dt_u &\leq \prod_{(u, v) \in E(H)} \Rnd{\int_{[0, 1]^2} f_G(t_u, t_v)^\D dt_u dt_v}^{1/\D} \\
                          & = \Frac{2m}{n^2}^{e/\D} \\
                                                                                                                                                                               & \leq C m^{e/\D} n^{-2e/\D},
        \end{align*}
        where we use the fact that $f_G$ is zero-one valued, so, $f_G^\D = f_G$. This finishes the proof.
    \end{proof}

    The graphon representation and the formula for homomorphism density is also useful in bounding the \emph{expected number of copies} of a particular graph $H$ in a random graph $G$ with independent edges. For a given collection of probabilities $0 \leq p_{uv} \leq 1$ for $u < v \in [n]$, consider the random graph model on the vertex set $[n]$ where the edge $(u, v)$ is present with probability $p_{uv}$, independently of other edges (undirected edges). For such a model we define:

    \begin{defn}[Graphon representation of a finite random graph $G$ with independent edges]
        For the model described above, the associated graphon is defined as a function $f_G : [0, 1]^2 \to [0, 1]$ satisfying
        \begin{align*}
            f_G(x, y) = p_{ij} \iff \exists\ 0 \leq i, j < n : \f{i}{n} \leq x < \f{i + 1}{n}, \f{j}{n} \leq y < \f{j + 1}{n}, (i, j) \in E(G)
        \end{align*}
        else it is zero.
    \end{defn}

    One of the key ingredients in our argument is the following novel result about the number of copies of $H$ in $G$ containing a given edge $e \in G$, which, as alluded to in section \ref{idea}, will be crucial in our analysis of the core, see \rlem{lem:product}. We are not aware of a similar previous result.
    \begin{lem}
        \label{lem:fixedholder}
        Let $\G$ be a subgraph of $K_n$ (with vertex set $[n]$), and let $f$ be the graphon associated with the random graph model $G$ on $[n]$ where the edges in $\G$ are present with probability 1, and every other edge is present with probability $p$ independently. Suppose $\G$ has a distinguished edge $(a, b)$, and let $H$ be a $\D$-regular graph with $q$ vertices. Then, the expected number of copies of $H$ in $G$ containing $(a, b)$ is at most:
        \begin{align*}
            \E \Q_H(G) \leq C(H) \Rnd{d_a + n p^\D}^{\f{\D - 1}{\D}}\Rnd{d_b + n p^\D}^{\f{\D - 1}{\D}}\Rnd{e + n^2 p^\D}^{q/2 - 2 + 1/\D}
        \end{align*}
        where $e = e(\G)$ and $d_a, d_b$ are the degrees of $a, b$ respectively in $\G$.

        \fix{In addition, if $\Q_H'(G)$ is the number of expected number of copies of $H$ in $G$ containing $(a, b)$ where some edge must come from outside $\G$, then
        \begin{align*}
        \E \Q'_H(G) & \leq C(H) \max(B_1, B_2, B_3)
        \intertext{where}
        B_1 & \df \Rnd{d_a + n p^\D}^{\f{\D - 2}{\D}}\Rnd{d_b + n p^\D}^{\f{\D - 1}{\D}}\Rnd{e + n^2 p^\D}^{q/2 - 2 + 1/\D} pn^{1/\D},\\
        B_2 & \df \Rnd{d_a + n p^\D}^{\f{\D - 1}{\D}}\Rnd{d_b + n p^\D}^{\f{\D - 2}{\D}}\Rnd{e + n^2 p^\D}^{q/2 - 2 + 1/\D} pn^{1/\D},\\
        B_3 & \df \Rnd{d_a + n p^\D}^{\f{\D - 1}{\D}}\Rnd{d_b + n p^\D}^{\f{\D - 1}{\D}}\Rnd{e + n^2 p^\D}^{q/2 - 2} pn^{2/\D}.
    \end{align*}}
    \end{lem}
    The terms $B_1, B_2, B_3$ correspond to the cases when the edge coming from outside $\G$ is adjacent to $a$, $b$ or none, respectively.
    \begin{proof}
        We begin by proving the first claim. Observe that
        \begin{align*}
            \Q_H(G) \leq \sum_{u \sim v} N(u, v),
        \end{align*}
        where 
        \begin{align}
            \label{eq:nuv}
            N(u, v) \df \sum_{\substack{\phi : V(H) \to [n] \\ \phi(u) = a, \phi(v) = b, \\ \phi \text{ injective}}} 
        \prod_{x \sim u} \1_{(\phi(u), \phi(x)) \in E(G)}\prod_{y \sim v}\1_{(\phi(v), \phi(y)) \in E(G)}\prod_{\substack{w \sim z \\ w, z \neq u, v}}\1_{(\phi(w), \phi(z)) \in E(G)}.
        \end{align}
        $N(u, v)$ counts all (injective) mappings $\phi : H \to [n]$ preserving edges in $H$, such that the edge $(u, v)$ maps to $(a, b)$.
    Here by $x \sim y$ we mean $(x, y) \in E(H)$ (for brevity, in this proof we also adopt the convention that for $u$, if we write $x \sim u$, we assume that $x \neq v$, and similarly for $v$). Therefore, it is sufficient to fix $u \sim v$ and prove the result, because summing over all distinct $u \sim v$ will only change the constant $C(H)$. For fixed $u \sim v$ we have
        \begin{align*}
            \E N(u, v) &\leq \sum_\phi \E[\prod_{x \sim u} \1_{(\phi(u), \phi(x)) \in E(G)}\prod_{y \sim v}\1_{(\phi(v), \phi(y)) \in E(G)}\prod_{\substack{w \sim z \\ w, z \neq u, v}}\1_{(\phi(w), \phi(z)) \in E(G)}],
            \intertext{where $\phi$ is understood to satisfy the conditions in  \eqref{eq:nuv}. Since $\phi$ is injective, all the edge random variables in the expectation above are independent, so we can write:}
                       &= \sum_\phi \prod_{x \sim u}\E[\1_{(\phi(u), \phi(x)) \in E(G)}]\prod_{y \sim v}\E[\1_{(\phi(v), \phi(y)) \in E(G)}]\prod_{\substack{w \sim z \\ w, z \neq u, v}}\E[\1_{(\phi(w), \phi(z)) \in E(G)}].
                       \intertext{For any vertex $\al \in [n]$, define the block of $\al$ as $B_\al \df [\al/n, (\al + 1)/n)$. Then note that, for instance, $\E[\1_{(\phi(x), \phi(y)) \in E(G)}] = f(t_x, t_y)$ for any $t_x \in B_{\phi(x)}, t_y \in B_{\phi(y)}$. For any collection $\{t_x \in B_{\phi(x)} : x \in H\}$, we then have that:}
                       &= \sum_\phi \prod_{x \sim u} f(t_u, t_x) \prod_{y \sim v} f(t_v, t_y) \prod_{\substack{w \sim z \\ w, z \neq u, v}} f(t_w, t_z).
                       \intertext{Fix $t_u \in B_a, t_v \in B_b$. Since the summand is a constant over all choices of $(t_x)_{x \neq u, v}$ as long as $(t_x)_{x \neq u, v} \in \prod_{x \neq u, v} B_{\phi(x)}$, the above sum is the same as}
                       &= n^{q - 2} \sum_\phi \int_{\prod_{x \neq u, v} B_{\phi(x)}} \prod_{x \sim u} f(t_u, t_x) \prod_{y \sim v} f(t_v, t_y) \prod_{\substack{w \sim z \\ w, z \neq u, v}} f(t_w, t_z) \prod_{x \neq u, v} dt_x,
                       \intertext{where the factor of $n^{q - 2}$ is because the volume of each block $B_{\phi(x)}$ is $n^{-1}$. Finally observe that across various choices of $\phi$, the domains $\prod_{x \neq u, v} B_{\phi(x)}$ are disjoint, so that the above is}
                       &\leq n^{q - 2} \int_{[0, 1]^{q - 2}} \prod_{x \sim u} f(t_u, t_x) \prod_{y \sim v} f(t_v, t_y) \prod_{\substack{w \sim z \\ w, z \neq u, v}} f(t_w, t_z) \prod_{x \neq u, v} dt_x.
                       \intertext{At this point we apply Finner's inequality with the sets given by $\{x\}$ for all $x \sim u$, $\{y\}$ for all $y \sim v$, and $\{w, z\}$ for all $w \sim z$ with $w, z \neq u, v$. We use the weights $p_i \equiv \D$ for all these sets. It is easy to verify that the conditions in \rlem{lem:finner} hold, and therefore:}
                       &\leq n^{q - 2} \prod_{x \sim u} \Rnd{\int f(t_u, t_x)^\D dt_x}^{1/\D}\prod_{y \sim v} \Rnd{\int f(t_v, t_y)^\D dt_y}^{1/\D} \\
                       &\hskip80pt \prod_{\substack{w \sim z \\ w, z \neq u, v}} \Rnd{\int f(t_w, t_z)^\D dt_w dt_z}^{1/\D}.
                       \intertext{In these products, the values for $x, y, w, z$ do not matter because the integrals are always over $[0, 1]$ for each variable, so we can simplify to (we drop the domains $[0, 1]$ and $[0, 1]^2$ from the notation):}
                       &=  n^{q - 2}\Rnd{\int f(t_u, t_x)^\D dt_x}^{\f{\D - 1}{\D}} \Rnd{\int f(t_v, t_y)^\D dt_y}^{\f{\D - 1}{\D}} \Rnd{\int f(t_w, t_z)^\D dt_w dt_z}^{\f{q}{2} - 2 + \f{1}{\D}},
                       \intertext{where we use the fact that the number of edges with no endpoint among $u, v$ is $q\D/2 - (2\D - 1) = q\D/2 - 2\D + 1$. As the last step, we use the fact that $f(t_u, t_x)$ is $1$ when $t_x \in B_c$ with $c \sim a$ in $\G$ and otherwise it is $p$. Similar observations also hold for $f(t_v, t_y)$ and $f(t_w, t_z)$. Therefore the above is
                       }
                       &\leq n^{q - 2}\Rnd{\f{d_a}{n} + p^\D}^{\f{\D - 1}{\D}}\Rnd{\f{d_b}{n} + p^\D}^{\f{\D - 1}{\D}} \Rnd{\f{2e}{n^2} + p^\D}^{\f{q}{2} - 2 + \f{1}{\D}}, \\
                       \intertext{because $f(t_u, t_x) = 1$ if and only if $\phi(u) \sim \phi(x) = a$ in $\G$ (otherwise it is $p$). A similar argument holds for the other two integrals. The factor of $2$ (in $2e$) in the last integral is because each edge in $\G$ is counted twice. Simplifying the above,}
                       &= \underbrace{n^{q - 2} \cdot n^{-2\f{\D - 1}{\D} - 2\Rnd{\f{q}{2} - 2 + \f{1}{\D}}}}_{= 1} \Rnd{d_a + np^\D}^{\f{\D - 1}{\D}}\Rnd{d_b + np^\D}^\f{\D - 1}{\D}\Rnd{2e + n^2p^\D}^{\f{q}{2} - 2 + \f{1}{\D}},
       \end{align*}
       which completes the proof by absorbing the factor of $2$ in front of $e$ into the constant $C(H)$.

       We now modify this to prove the second claim. Similar to \eqref{eq:nuv}, define
       \begin{align*}
           N^{\al, \be}(u, v) & \df \sum_{\substack{\phi : V(H) \to [n] \\ \phi(u) = a, \phi(v) = b, \\ \phi \text{ injective}}} 
           \prod_{x \sim u} \1_{(\phi(u), \phi(x)) \in E(G)}\prod_{y \sim v}\1_{(\phi(v), \phi(y)) \in E(G)} \\ 
           & \hskip80pt \prod_{\substack{w \sim z \\ w, z \neq u, v}}\1_{(\phi(w), \phi(z)) \in E(G)} \cdot \1_{(\phi(\al), \phi(\be)) \notin E(\G)}
       \end{align*}
       for all $(\al, \be) \in E(H)$ such that $(\al, \be) \neq (u, v)$. This counts the same quantity as $N(u, v)$, except now, $(\al, \be)$ maps to an edge outside $\G$. Observe that
       \begin{align*}
           \Q'_H(G) \leq \sum_{(\al, \be) \in E(H)} \sum_{E(H) \ni (u, v) \neq (\al, \be)} N^{\al, \be}(u, v)
       \end{align*}
       so that it is again sufficient to prove the bound in question for $N^{\al, \be}(u, v)$ for fixed $\al, \be, u, v \in H$ instead of $\Q'_H(G)$. Using calculations similar to the above we then obtain that
       \begin{align*}
           \E N^{\al, \be}(u, v) \leq n^{q - 2}\int_{[0, 1]^{q - 2}} \prod_{x \sim u} f(t_u, t_x) \prod_{y \sim v} f(t_v, t_y) \prod_{\substack{w \sim z, \\ w, z \neq u, v}} f(t_w, t_z) \cdot \1_{f(t_\al, t_\be) = p} \prod_x dt_x,
       \end{align*}
       because $f(t_\al, t_\be) = p$ if and only if $(\phi(\al), \phi(\be)) \notin E(\G)$.
       As indicated in the remark after the statement, now there are several cases depending on if $\al = u$, $\al = v$ or $\al, \be \neq u, v$ (the cases for $\be$ being similar).

        $\bullet$ $\al = u, \be \neq v$: The integral is then bounded by Finner's inequality similar to the above to obtain

               \begin{align*}
                   \E N^{\al, \be}(u, v) & \leq n^{q - 2} \prod_{\be \neq x \sim u} \Rnd{\int f(t_u, t_x)^\D dt_x}^{1/\D} \prod_{y \sim v} \Rnd{\int f(t_v, t_y)^\D dt_y}^{1/\D} \\ 
                                         & \hskip80pt \prod_{\substack{w \sim z \\ w, z \neq u, v}} \Rnd{\int f(t_w, t_z)^\D dt_w dt_z}^{1/\D}
                                         \Rnd{\int f(t_u, t_\be)^\D \1_{f(t_u, t_\be) = p} dt_\be}^{1/\D}. \\
                                         \intertext{Observe that the last factor is $\leq p$, so the above is bounded by}
                                        &\leq n^{q - 2}\Rnd{\f{d_a}{n} + p^\D}^{\f{\D - 2}{\D}}\Rnd{\f{d_b}{n} + p^\D}^{\f{\D - 1}{\D}} \Rnd{\f{2e}{n^2} + p^\D}^{\f{q}{2} - 2 + \f{1}{\D}} \cdot p\\
                                        & = \Rnd{d_a + np^\D}^{\f{\D - 2}{\D}}\Rnd{d_b + np^\D}^{\f{\D - 1}{\D}}\Rnd{2e + n^2p^\D}^{\f{q}{2} - 2 + \f{1}{\D}}pn^{1/\D}.
               \end{align*}
        $\bullet$ $\al = v, \be \neq u$: A calculation similar to the previous case will yield:
                \begin{align*}
                    \E N^{\al, \be}(u, v) \leq \Rnd{d_a + np^\D}^{\f{\D - 1}{\D}}\Rnd{d_b + np^\D}^{\f{\D - 2}{\D}}\Rnd{2e + n^2p^\D}^{\f{q}{2} - 2 + \f{1}{\D}}pn^{1/\D}.
                \end{align*}
        $\bullet$ $\al, \be \neq u, v$: Then,
                \begin{align*}
                    \E N^{\al, \be}(u, v) & \leq n^{q - 2} \prod_{x \sim u} \Rnd{\int f(t_u, t_x)^\D dt_x}^{1/\D} \prod_{y \sim v} \Rnd{\int f(t_v, t_y)^\D dt_y}^{1/\D} \\
                                        &\hskip60pt \prod_{\substack{w \sim z \\ w, z \neq u, v \\ (w, z) \neq (\al, \be)}} \Rnd{\int f(t_w, t_z)^\D dt_w dt_z}^{1/\D} 
                                          \Rnd{\int f(t_\al, t_\be)^\D \1_{f(t_\al, t_\be) = p} dt_\al dt_\be}^{1/\D}. \\
                                         \intertext{Again the last factor is at most $p$, so that the above is bounded by}
                                         &\leq n^{q - 2}\Rnd{\f{d_a}{n} + p^\D}^{\f{\D - 1}{\D}}\Rnd{\f{d_b}{n} + p^\D}^{\f{\D - 1}{\D}} \Rnd{\f{2e}{n^2} + p^\D}^{\f{q}{2} - 2} \cdot p\\
                                        & = \Rnd{d_a + np^\D}^{\f{\D - 1}{\D}}\Rnd{d_b + np^\D}^{\f{\D - 1}{\D}}\Rnd{2e + n^2p^\D}^{\f{q}{2} - 2}pn^{2/\D}.
                \end{align*}
       Taking a maximum over these three cases and modifying $C = C(H)$, we obtain the desired result.
    \end{proof}
}
\section{Few copies}
\label{sec:few}
In this section, we will prove our main theorem for ``small'' values of $k_n$. Henceforth we will abbreviate $Q = \Q_H(G)$.

\begin{defn}[Spanned Graph]
    A graph $G$ is called an {\bf $H$-spanned graph} (or ``spanned graph'' for short) if $G$ is connected and each edge is contained in a copy of $H$ in $G$. We say it is spanned by $\ell$ copies of $H$ if there are $\ell$ copies of $H$ in $G$ such that each edge of $G$ is contained in one of these copies. We will often use $\ell_* = \ell_*(G)$ to indicate the smallest possible value of $\ell$ given $G$. 

    For brevity, we will often say that a graph is spanned by $\ell$ copies of $H$ if it is a spanned graph and it is spanned by $\ell$ copies of $H$.
\end{defn}

    This is the analog of ``triangle-induced subgraphs'' as defined in \cite{ganguly}. Also note that $\ell_* = 1$ if and only if $G = H$ up to isomorphism.

Next we define some useful events:
\begin{itemize}
    \item $F_\ell$: there is a spanned subgraph of $G$ which is spanned by $\ell$ copies of $H$.
    \item $D_s$: there are $s$ vertex-disjoint copies of $H$ in $G$ \emph{(for brevity we will write ``disjoint'' to mean vertex-disjoint)}.
    \item $E^s_{\ell_1, \ldots, \ell_m}$: there are (at least) $s + m$ disjoint spanned subgraphs of $G$, with $s$ of them containing exactly one copy of $H$, and the remaining $m$ of them each spanned by $\ell_i$ copies for $1 \leq i \leq m$. Since the order of the $\ell_i$ are immaterial, we will always assume that $2 \leq \ell_1 \leq \ell_2 \leq \ldots \leq \ell_m$. When $s = 0$, we drop $s$, and simply write $E_{\ell_1, \ldots, \ell_m}$. (Here $s$ stands for ``singleton'' copies, and $m$ stands for ``multiple'' copies)
\end{itemize}

\fixr{Define/Cite disjoint occurrence operator.}

Using $\circ$ for the ``disjoint occurrence'' operator on events (see \cite[Section 3]{bk} for the definition), it is easy to see that 
\begin{align*}
    E^s_{\ell_1, \ldots, \ell_m} = D_s \circ F_{\ell_1} \circ \cdots \circ F_{\ell_m},
\end{align*}
and so 
\begin{align}
    \label{eq:bk}
    \P(E^s_{\ell_1, \ldots, \ell_m}) \leq \P(D_s)\prod_{i = 1}^m \P(F_{\ell_i}),
\end{align}
due to the BK-inequality, see \cite{bk}. To see why such a bound is useful, observe that if a graph $G$ satisfies $\Q = \Q_H(G) \geq k_n$, one can drop edges in $G$ not contained in any copy of $H$, and decompose the remaining graph into connected components. Clearly each remaining component is $H$-spanned, and no copies of $H$ were deleted, so that there are numbers $s, \ell_1, \ldots, \ell_m$ such that $G \in E^s_{\ell_1, \ldots, \ell_m}$. The rest of this section will be devoted to efficiently union bounding over various choices of these numbers. But first we bound the probabilities of $D_s$ and $F_\ell$ individually.
\begin{lem}\label{lem:disjointprob}
    $\P(D_s) \leq C\Exp{- s \log s}$ for some $C = C(H)$.
\end{lem}

\begin{proof}
    The number of ways we can pick $s$ disjoint subsets of $[n]$ is at most
    \begin{align*}
        \f{1}{s!}\binom{n}{q}\binom{n - q}{q} \cdots \binom{n - (s - 1)q}{q} = \f{n!}{s! (q!)^s (n - qs)!} \leq \f{n^{qs}}{s!(q!)^s}.
    \end{align*}
    Given a fixed set of vertices $S$ of size $q$, the probability of it having a copy of $H$ is at most
    \begin{align*}
        C p_n^{-e} = C n^{-2e/\D} = C n^{-q},
    \end{align*}
    where the constant $C = C(H)$ is due to the number of isomorphic copies of $H$ in $K_q$. Then by a union bound, the probability of $D_s$ satisfies
    \begin{align*}
        \P(D_s) \leq C\exp(qs \log n - s \log(q!) - \log(s!) - qs\log n) = C\exp(-s \log (q!) - \log (s!)).
    \end{align*}
    By Stirling's formula, we have $\log(s!) = s \log s - s + O(\log s)$. Therefore 
    \begin{align*}
        s\log(q!) + s\log s - s + O(\log s) \geq s\log s, \quad \forall s \geq 1,
    \end{align*}
    (using $q \geq 3$ and $\log 6 > 1$) finishing the proof.
\end{proof}
    

The next lemma delivers the control on $\f{2}{\D}e(S) - v(S)$, serving as a replacement for the ``tree-excess edges'' interpretation of $e(S) - v(S) + 1$ in \cite{ganguly} as described in the idea of proof section. This will be useful for furnishing a bound on $\P(F_\ell)$.

\begin{lem}
    \label{lem:spannedstruct}
    If $S$ is a graph spanned by $\ell \geq 2$ copies of $H$ then $\f{2}{\D} e(S) - v(S) \geq C e(S)$, where $C = C(H)$.
\end{lem}
\begin{proof}
    Observe that the optimal number of copies $\ell_* = \ell_*(S)$ needed to span $S$ satisfies $\ell_* \geq Ce(S)$ where $C = 1/q^2$, because each copy of $H$ contains at most $q^2$ edges. Order these copies of $H$ in $S$ as $H_1, \ldots, H_{\ell_*}$, 
    such that for each $i$, $V(H_i) \cap V(S_{i - 1}) \neq \nul$, where $S_i$ is the subgraph of $S$ spanned by $H_1, \ldots, H_i$ (this can be done because $S$ is spanned and connected). Note that in general $S_i$ is distinct from the subgraph of $S$ induced by $V(S_i)$. Each edge incident on any $x \in V(H_i) - V(S_{i - 1})$ comes from $H_i$, and hence the degree of any such $x$ is exactly $\D$ in $S_i$. For each $i > 1$, also write $b_i \geq 1$ for the number of edges between $V(H_i) - V(S_{i - 1})$ and $V(S_{i - 1})$ (the number of boundary edges), and let $v_i = |V(H_i) - V(S_{i - 1})|$ (the number of new vertices). See Figure \ref{fig:deterministic_few} below for an illustration of these definitions.

    \begin{figure}[H]
        \centering
    \def\svgwidth{0.4\columnwidth}
    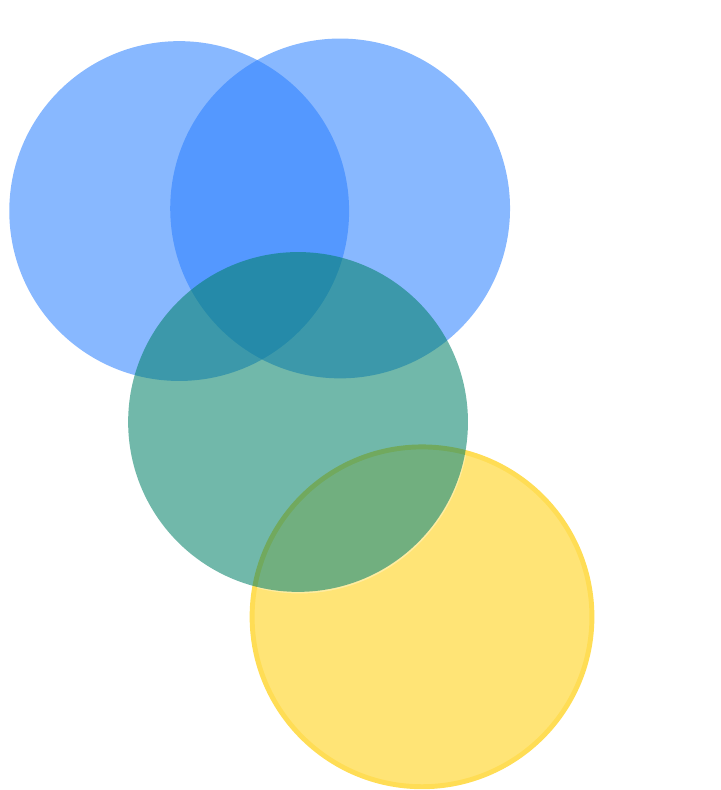

        \caption{Variables in the proof of \rlem{lem:spannedstruct} with four copies of $H$.}
        \label{fig:deterministic_few}
    \end{figure}

    We track the quantity $f_i \df \f{2}{\D} e(S_i) - v(S_i)$. Observe that $e(S_i) \geq e(S_{i - 1}) + b_i + e(H_i - S_{i - 1})$, where $H_i - S_{i - 1}$ is the subgraph of $S_i$ induced by $V(H_i) - V(S_{i - 1})$ (the inequality is due to ignoring edges in $H_i$ with both endpoints in $V(S_{i - 1})$). Crucially, note that
    \begin{align*}
        2e(H_i - S_{i - 1}) + b_i = \D v_i
    \end{align*}
    because both sides count the edges within $V(H_i) - V(S_{i - 1})$ twice. Therefore

    \begin{align*}
        e(S_i) \geq e(S_{i - 1}) + b_i + \f{\D v_i - b_i}{2} = e(S_{i - 1}) + \f{b_i}{2} + \f{\D v_i}{2},
    \end{align*}
    and so,
    \begin{align*}
        f_i = \f{2}{\D}e(S_i) - v(S_i) & \geq \f{2}{\D} e(S_{i - 1}) + \f{b_i}{\D} + v_i - v(S_i) \\
                                       & = \f{2}{\D}e(S_{i - 1}) + \f{b_i}{\D} + v_i - v(S_{i - 1}) - v_i \\
                                       \intertext{(substituting $v(S_i) = v(S_{i - 1}) + v_i$). Therefore, the above is}
                                       & = f_{i - 1} + \f{b_i}{\D} \geq f_{i - 1} + \f{1}{\D}\quad \text{(since $b_i \geq 1$).}
    \end{align*}
    Since $f_1 = 0$ we get $f_{\ell_*} \geq \f{\ell_* - 1}{\D} \geq \f{\ell_*}{2\D}$ as long as $\ell_* \geq 2$ which holds if $\ell \geq 2$ (this is because $\ell \geq 2$ implies that there are more than $e(H)$ edges in $S$). To conclude the proof, recall that $\ell_* \geq e(S)/q^2 = Ce(S)$.
\end{proof}

Using this structural result, we can now provide an upper bound on $\P(F_\ell)$.

\begin{lem}
    $\P(F_\ell) \leq \exp(-C \ell^{2/q} \log n)$ for all $2 \leq \ell \leq n^c$ and all sufficiently large $n$, where $c = c(H), C = C(H)$ are constants.
\end{lem}
\begin{proof}
    For a fixed graph $S$ spanned by $\ell$ copies of $H$, the probability of its occurrence in $G \sim \cG(n, p = n^{-2/\D})$ is at most
    \begin{align*}
        \binom{n}{v(S)} p^{e(S)} \leq n^{v(S) - \f{2}{\D}e(S)},
    \end{align*}
    and as we saw above in \rlem{lem:spannedstruct}, for such a graph, $\f{2}{\D}e(S) - v(S) \geq C e(S)$. Therefore this probability is at most $\exp(-Ce(S)\log n)$.

    Let $F_{\ell, v, e}$ be the event that there is a subgraph of $G$ with $v$ vertices, $e$ edges, and spanned by $\ell$ copies of $H$ (note that $v \leq q\ell$). The number of such $S$ is at most $\binom{v^2}{e} \leq v^{2e}$, so by a direct union bound,
    \begin{align*}
        \P(F_{\ell, v, e}) & \leq \exp(-Ce\log n) v^{2e}  \\
                           & \leq \exp(-Ce\log n) \exp(2e \log (q\ell)),\\
                           & \leq \exp(-Ce\log n),
                           \intertext{by a different choice of $C$, as long as $\ell \leq n^c$ for some small $c$ depending on $H$. Finally, using the fact that $\ell \leq C e^{q/2}$ (using \rlem{lem:holder}) we see that the above is}
                           & \leq \exp(-C\ell^{2/q}\log n).
    \end{align*}

    For a graph spanned by $\ell$ copies of $H$ we have $v, e \leq C\ell$, so by union bounding over all such possibilities, we get
    \begin{align*}
        \P(F_\ell) \leq C\ell^2 \exp(-C\ell^{2/q}\log n) \leq \exp(-C\ell^{2/q}\log n),
    \end{align*}
    (up to changing $C$) for sufficiently large $n$.
\end{proof}

As alluded to earlier, we will now combine the estimates above using a union bound. The following definition will be useful.
\begin{defn}
    \label{defn:etilde}
    For a tuple $(c_1, c_2, \ldots, c_t)$, define $\tE_{c_1, \ldots, c_t}$ to be the event that in the random graph $G \sim \cG(n, p_n)$, there are $c_1 + \cdots + c_t$ disjoint spanned graphs, with $c_i$ of them spanned by $2^i$ copies of $H$, for all $i$.
\end{defn}
Observe that $\tE_{c_1, \ldots, c_t}$ is a special case of events of the form $E_{\ell_1, \ldots, \ell_m}$ where $m = \sum_{i = 1}^t c_i$ and the $\ell_i$s are powers of 2. Also, if a graph $G$ satisfies $E_{\ell_1, \ldots, \ell_m}$, then there is a tuple $(c_1, \ldots, c_t)$ such that it also satisfies $\tE_{c_1, \ldots, c_t}$, with
    \begin{align*}
        \sum_{i = 1}^{t} c_i 2^i \geq \f{1}{2}\sum_{j = 1}^m \ell_j.
    \end{align*}
    This is essentially done via ``dropping'' copies of $H$ till the components are spanned by $2^i$ components for some $i$. More precisely, for each valid $i$, define
    \begin{align}
        \label{eq:cibound}
        c_i \df |\{ j : 2^i \leq \ell_j < 2^{i + 1} \}|. 
    \end{align}
    We claim that $G \in \tE_{c_1, \ldots, c_t}$. Let $U_1, \ldots, U_m$ be some choice of disjoint spanned subgraphs of $G$ witnessing $E_{\ell_1, \ldots, \ell_m}$. Choose $j$ such that $U_j$ is spanned by $\ell_j$ copies of $H$ and let $i$ be such that $2^i \leq \ell_j < 2^{i + 1}$. Order these copies of $H$ such that the subgraph induced by any number of copies from the beginning is connected (which can always be done because $U_j$ is connected). Then take the union of the first $2^i$ copies. By construction this is a connected subgraph spanned by $2^i$ copies of $H$. Doing this for each $j$ proves the claim.

    As we will shortly see, $\P(E_{\ell_1, \ldots, \ell_m})$ is comparable to $\P(\tE_{c_1, \ldots, c_t})$, whereas the number of possible choices is a lot less for $(c_1, \ldots, c_t)$ than for $(\ell_1, \ldots, \ell_m)$, facilitating an efficient union bound.

\begin{lem}
    \label{lem:few}
    For $2 \leq k_n \leq n^c$ (for a sufficiently small $c = c(H)$), we have
    \begin{align*}
        \P(\Q \geq k_n) \leq C\Exp{-C' \min(k_n \log k_n, k_n^{2/q}\log n)},
    \end{align*}
    where $C = C(H), C' = C'(H)$ are constants depending only on $H$.
\end{lem}

As described earlier, this lemma demonstrates the transition in tail behavior promised in our main theorem. However, this result only holds for $k_n$ at most a small polynomial of $n$, and one must show that there are \emph{no other transitions}, i.e., this behavior persists all the way up to the maximum possible value of $k_n$. This is achieved in the next section after this proof.

\begin{proof}
    We abbreviate $k = k_n$ throughout this proof. Recall from the discussion below \eqref{eq:bk} that the event $\{\Q \geq k\}$ is equivalent to 
    \begin{align*}
        \bigcup_{\substack{s, m \geq 0, \ell_i \geq 2 \\ s + \sum_{i \leq m} \ell_i \geq k}} E^s_{\ell_1, \ldots, \ell_m}.
    \end{align*}
    However any graph satisfying this also satisfies the event
    \begin{align*}
        \bigcup_{\substack{s, m \geq 0, \ell_i \geq 2 \\ s + \sum_{i \leq m} \ell_i = k}} E^s_{\ell_1, \ldots, \ell_m},
    \end{align*}
    because a spanned graph spanned by $\ell$ copies of $H$ has a subgraph spanned by $\ell'$ copies of $H$ for any $\ell' \leq \ell$ by reasoning as below \eqref{eq:cibound}. We therefore focus on bounding the probability of the latter.

    Using the BK inequality \eqref{eq:bk} and the bound on $\P(F_\ell)$ we obtain
    \begin{align*}
        \P(\tE_{c_1, \ldots, c_t}) &\leq \prod_{i = 1}^t \P(F_{2^i})^{c_i} \\
                                   &\leq \Exp{-C \log n\sum_{i = 1}^t c_i 2^{2i/q}}
    \end{align*}

    We use this bound and the fact that for each collection $\ell_1, \ldots, \ell_m$ (with each $\ell_i \geq 2$), we can find a vector $c_1, \ldots, c_t$ with 
    \begin{align}
        \sum_{i \leq m } \ell_i \geq \sum_{j \leq t} c_i 2^i \geq \f{1}{2}\sum_{i \leq m} \ell_i,
        \label{eq:condc}
    \end{align}
    such that $E_{\ell_1, \ldots, \ell_m} \subseteq \tE_{c_1, \ldots, c_t}$. This allows us to conclude that for any $k$,
    \begin{align*}
        \P[\bigcup_{\substack{m \geq 1, \ell_i \geq 2 \\ \sum_{i \leq m} \ell_i = k}} E_{\ell_1, \ldots, \ell_m}] & \leq 
        \P[\bigcup_{\substack{t \geq 1, c_i \geq 0 \\ k \geq \sum_i c_i 2^i \geq  k/2}} \tE_{c_1, \ldots, c_t}] \\
                                                                                                                  & \leq \exp((\log k)^2)\Exp{-C\log n \cdot \min \sum_{i \geq 1} c_i 2^{2i/q}},
                                                                                                                  \intertext{where the $\min$ is over all possible $c_i$s satisfying the condition in \eqref{eq:condc}. The $\exp((\log k)^2)$ term is a simple upper bound on the number of possible tuples $(c_1, \ldots, c_t)$ where $t \leq \log_2 k$ and $0 \leq c_i \leq k/2$. By an elementary inequality (see \rlem{lem:concave} in the Appendix) we see that the above is }
                                                                                                                  &\leq \Exp{C\Rnd{\log^2 k - k^{2/q} \log n}}.
    \end{align*}
    Now observe that $E^s_{\ell_1, \ldots, \ell_m} = D_s \circ E_{\ell_1, \ldots, \ell_m}$. Another application of the BK inequality and our estimate of $\P(D_s)$ from \rlem{lem:disjointprob} yields
    \begin{align*}
        \P[\bigcup_{\substack{s, m \geq 0, \ell_i \geq 2 \\ s + \sum_{i \leq m} \ell_i = k}} E^s_{\ell_1, \ldots, \ell_m}] & 
        \leq \sum_{s \leq k} \P(D_s) \P[{\bigcup_{\substack{m \geq 0, \ell_i \geq 2 \\ \sum_{i \leq m} \ell_i = k - s}} E_{\ell_1, \ldots, \ell_m}}] \\
                                                                                                                           &\leq C\exp(\log^2 k + \log k)\Exp{-C\min_{1 \leq s \leq k} \Rnd{s\log s + (k - s)^{2/q}\log n}},
                                                                                                                           \intertext{at which point we invoke \rlem{lem:competition} from the Appendix to obtain that the above is }
                                                                                                                           &\leq C\exp(2\log^2 k)\Exp{-C\min(k \log k, k^{2/q}\log n)}.
    \end{align*}
    Finally observe that since $2\log^2 k \ll \min(k \log k, k^{2/q}\log n)$ as $k \to \infty$, by modifying constants we may drop the prefactor of $\exp(2\log^2 k)$, yielding the required result.
\end{proof}


\section{Many copies}
\label{sec:many}

In this section we will consider the case of large $k_n$ and demonstrate that the behavior described in \rlem{lem:few} actually persists all the way up to the maximum possible value of $k_n$. However our results here will only hold for $k_n$ such that $k_n \gg f(n)$ for all $f(n) = \poly\log n$ (we will assume this lower bound on $k_n$ throughout this section). But this suffices to complete the proof of our main theorem since \rlem{lem:few} holds even for $k_n$ that are small polynomial powers of $n$.

\fix{Our arguments here will adapt a version of the techniques from \cite{samotij} and hence is related to the ideas in \cite{suman}. The main idea is to consider a class of subgraphs (called ``cores'') which are present with high probability in $G$ when conditioned to have $k_n$ triangles. Once such a subgraph is planted in $G$, the expectation of $\Q$ is almost as much as the required number of $k_n$. As in \cite{samotij}, before defining a ``core'' we define a relaxed version called a ``seed''.}

Let $w_n = 1/\log n$ in the sequel. In principle, $w_n$ can be taken to be anything that is $1/\poly\log n$, but we fix it so for concreteness.

\begin{defn}[Seed]
    \label{def:seed}
    Let $S \subseteq K_n$ be a (labeled) graph on the vertex set $V = [n]$. We call $S$ a seed if and only if:
    \begin{enumerate}
        \item  $\E_S \Q \geq (1 - w_n)k_n$,
        \item  $e(S) \leq C_s w_n^{-1}k_n^{2/q}\log(1/p_n)$,
    \end{enumerate}
    where $e(S)$ is the number of edges in $S$. Here $\E_S \Q = \E[Q \mid S \subseteq K_n]$ is the conditional expectation of the number of copies of $H$ in $G$ given that the edges in $S$ are present in $G$, and $C_s$ is a constant depending on $H$ that will be fixed later.
\end{defn}

Following the arguments of \cite{samotij} one can deduce the following.
\begin{lem}[Must have a seed]
    \label{lem:seed}
    There is a sequence $\xi_n \to 0$ such that
    \begin{align*}
        \P(\Q \geq k_n) \leq (1 + \xi_n)\P(\text{$G$ has a seed})
    \end{align*}
    where $G \sim \cG(n, p_n)$.
\end{lem}
\begin{proof}
    For a subgraph $S \subseteq K_n$, set $Z_S \df \1_{\{G \cap S \text{ has no seed}\}}$ (we will also let $Z \df Z_{K_n}$) and for any $\H$, a subset of edges of $K_n$ forming an isomorphic copy of $H$, denote $Y_\H \df \1_{\{\text{all edges in }\H\text{ are present in }G\}}$. Then observe that 
    \begin{align}
        \label{eq:qbndorig}
        \P(Q \geq k_n, \text{$G$ has no seed}) = \P(\Q Z \geq k_n) \leq \f{\E[\Q^\ell Z]}{k_n^\ell},
    \end{align}
    for any $\ell \geq 1$ (we will choose $\ell$ later). Since $\Q = \sum_\H Y_\H$ (sum over all possible edge subsets of $K_n$ which are isomorphic to $H$).
    \begin{align*}
        \E[\Q^\ell Z] &= \sum_{\H_1,  \ldots, \H_\ell} \E[Y_{\H_1}\cdots Y_{\H_\ell} \cdot Z] \\
                        &\leq \sum_{\H_1, \ldots, \H_\ell} \E[Y_{\H_1}\cdots Y_{\H_\ell} \cdot Z_{\H_1 \cup \ldots \cup \H_{\ell - 1}}], \qquad\qquad \text{because $Z_S \geq Z_T$ if $S \subseteq T$},\\
                        &= \sum_{\H_1, \ldots, \H_\ell} \P(\H_1, \ldots, \H_\ell \text{ are present}, \H_1 \cup \ldots \cup \H_{\ell - 1} \text{  has no seed}), 
                        \intertext{But $\H_1 \cup \ldots \cup \H_{\ell - 1}$ having (or not having) a seed is a deterministic fact, so we can reduce the sum over all collections of $\H_i$ where this holds. We can group the terms in the sum above by $\H_1, \ldots, \H_{\ell - 1}$ to obtain}
                        &= \sum_{\H_1, \ldots, \H_{\ell - 1} \text{ with no seed }} \P(\H_1, \ldots, \H_{\ell - 1} \text{ are present}) \sum_{\H_\ell} \E[Y_{\H_\ell} \mid Y_{\H_i} = 1, \forall i \leq \ell - 1] \\
                        &= \sum_{\H_1, \ldots, \H_{\ell - 1} \text{ with no seed }} \E[Y_{\H_1}\cdots Y_{\H_{\ell - 1}} Z_{\H_1 \cup \ldots \cup \H_{\ell - 1}}] \underbrace{\E_{\H_1 \cup \ldots \cup \H_{\ell - 1}} \Q}_{\leq (1 - w_n)k_n},
    \end{align*}
    where we use the fact that $\H_1 \cup \ldots \cup \H_{\ell - 1}$ has no seed, and in particular, is not a seed itself. But we can only apply this reasoning if the total number of edges in $\H_1 \cup \ldots \cup \H_{\ell - 1}$ is at most $C w_n^{-1} k_n^{2/q}\log(1/p_n)$ (see Definition \ref{def:seed}) which happens at least as long as 
    \begin{align}
        \label{eq:bndl}
     \ell \leq C_s q^{-2} w_n^{-1} k_n^{2/q} \log(1/p_n),
    \end{align}    
    because $\H_1 \cup \ldots \cup \H_{\ell - 1}$ has at most $q^2\ell$ edges (and thus the only reason it is not a seed is that the conditional expectation condition is violated). Observe that the bound obtained for $\E[Q^\ell Z]$ has a recursive structure, which we can iterate in $\ell$ to obtain
   \begin{align*}
        \E[Q^\ell Z] \leq (1 - w_n)^\ell k_n^\ell,
    \end{align*}
    giving us the eventual upper bound via \eqref{eq:qbndorig}
    \begin{align*}
        \P(Q \geq k_n, \text{$G$ has no seed}) \leq (1 - w_n)^\ell
    \end{align*}
    for all $\ell$ satisfying the bound in \eqref{eq:bndl}.
    Now observe that a clique of size $c k_n^{1/q}$ has at least $k_n$ copies of $H$ (where $c = c(H)$ is a constant), and it has $C k_n^{2/q}$ many edges. Therefore by choosing $C_s$ in the definition of a seed large enough, we can make this clique a seed. The probability of its occurrence is
    \begin{align*}
        \P(\text{$c k_n^{1/q}$ clique in $G$}) \geq \Exp{-C k_n^{2/q}\log n}
    \end{align*}
    (recall that various occurrences of $C$ can mean different $H$-dependent constants).
    Choose $\ell = \floor{C_sq^{-2}w_n^{-1}k_n^{2/q}\log(1/p_n)}$ to get:
    \begin{align*}
        \P(Q \geq k_n, \text{$G$ has no seed}) &\leq (1 - w_n)^\ell \\
                                                 &\leq \Exp{-C q^{-2}k_n^{2/q}\log(1/p_n)} \\
                                                 &\leq \xi_n \P(c k_n^{2/q} \text{ clique})
                                                 \intertext{where $\xi_n \to 0$ as long as $C_s/q^2$ is large enough. But as we saw above, this is a seed, so we can replace the bound above by}
                                                 &\leq \xi_n \P(G_n \text{ has seed}),
    \end{align*}
    which finishes the proof. 
\end{proof}

The previous lemma reduces our task to upper bounding $\P(G \text{ has a seed})$. This is essentially achieved via union-bounding over all possible seeds. But the number of possible seeds is extremely high, so we instead consider a more restrictive structure called ``cores'', as defined below. Each seed contains a core, so we have a version of the preceding lemma with $\{\text{$G$ has a seed}\}$ replaced by $\{\text{$G$ has a core}\}$.

\begin{defn}[Core]
    \label{def:core}
    Let $\G \subseteq K_n$ be a (labeled) graph on $V = [n]$. We call it a core if it has no isolated vertices and,
    \begin{enumerate}
        \item  $\E_\G Q \geq (1 - 2w_n)k_n$,
        \item  $e(\G) \leq C_sw_n^{-1}k_n^{2/q}\log(1/p_n)$,
        \item  $\E_\G Q - \E_{\G - f} Q \geq t_n \df \displaystyle\f{w_n^2 k_n^{(q - 2)/q}}{C_s \log (1/p_n)} \Rnd{= \Om(k_n^{(q - 2)/q})}$, for all edges $f \in E(\G)$,
    \end{enumerate}
    that is, it is a (slightly weaker) seed, but now each edge has to contribute significantly.
\end{defn}

As indicated earlier, every seed contains a core:

\begin{lem} If $G$ is a seed, then there is a subset $\G \subseteq G$ which is a core.
\end{lem}
\begin{proof}
    We iteratively delete edges violating (3). Removing such an edge does not reduce $\E_G Q$ by more than $t_n$. Therefore over all $e(G)$ edges, the maximum possible decrease is $t_n e(G) \leq w_nk_n$, so that the reduced graph still satisfies (1). If we cannot remove an edge, it is already a core.
\end{proof}

The strict restrictions on cores allow us to prove a few crucial structural results about them.

\begin{lem}[Cores have many copies of $H$]
    \label{lem:manycopies}
    For a core $\G$, let $\Q_H(\G)$ be the number of copies of $H$ in $\G$. Then $\Q_H(\G) \geq k_n - o(k_n)$. In particular, via \rlem{lem:holder}, there is a $C$ such that $e(\G) \geq C k_n^{2/q}$.
\end{lem}
\begin{proof}
    We denote by $\H$ an isomorphic copy of $H$ in $K_n$, i.e. a subset of edges which are isomorphic to $H$. Then we know that
    \begin{align*}
        (1 - 2w_n)k_n \leq \E_\G \Q = \sum_{\H} p_n^{|\H - \G|} \leq \Q_H(\G) + \sum_{|\H - \G| \geq 1} p_n^{|\H - \G|},
    \end{align*}
    where $\H - \G$ is the collection of edges in $\H$ that are not in $\G$.
    For a given copy $\H$ in $K_n$, fix an isomorphism $\phi_\H$ mapping $E(H) \to \H$. Let $c$ be the number of edges in $\H - \G$ with both endpoints in $V(\G)$, $b$ be the number of edges in $\H - \G$ with one endpoint in $V(K_n) - V(\G)$ and other in $V(\G)$, and let $x$ be the number of vertices in $\H$ outside $\G$. 
    Also denote by $X \subseteq H$ the subgraph which is mapped (under $\phi_\H$) to outside $\G$ (so that $x = |V(X)|$), and let $Y \subseteq H$ be the subgraph of $H$ whose vertices are mapped (under $\phi_\H$) to $\G$ and whose edges are mapped to $E(\G)$ (so that $V(X) \cup V(Y) = V(H)$ but $E(X) \cup E(Y)$ is not $E(H)$; the edges counted in $b$ and $c$ are the ones in $E(H) - (E(X) \cup E(Y))$). Observe that $Y$ may contain isolated vertices. Denote by $y = |V(Y)|$. See Figure \ref{fig:structure_of_cores} below for an illustration of these definitions.     
    \begin{figure}[H]
        \centering
    \def\svgwidth{0.4\columnwidth}
\begingroup%
  \makeatletter%
  \providecommand\color[2][]{%
    \errmessage{(Inkscape) Color is used for the text in Inkscape, but the package 'color.sty' is not loaded}%
    \renewcommand\color[2][]{}%
  }%
  \providecommand\transparent[1]{%
    \errmessage{(Inkscape) Transparency is used (non-zero) for the text in Inkscape, but the package 'transparent.sty' is not loaded}%
    \renewcommand\transparent[1]{}%
  }%
  \providecommand\rotatebox[2]{#2}%
  \newcommand*\fsize{\dimexpr\f@size pt\relax}%
  \newcommand*\lineheight[1]{\fontsize{\fsize}{#1\fsize}\selectfont}%
  \ifx\svgwidth\undefined%
    \setlength{\unitlength}{222.16719259bp}%
    \ifx\svgscale\undefined%
      \relax%
    \else%
      \setlength{\unitlength}{\unitlength * \real{\svgscale}}%
    \fi%
  \else%
    \setlength{\unitlength}{\svgwidth}%
  \fi%
  \global\let\svgwidth\undefined%
  \global\let\svgscale\undefined%
  \makeatother%
  \begin{picture}(1,0.98288956)%
    \lineheight{1}%
    \setlength\tabcolsep{0pt}%
    \put(0,0){\includegraphics[width=\unitlength,page=1]{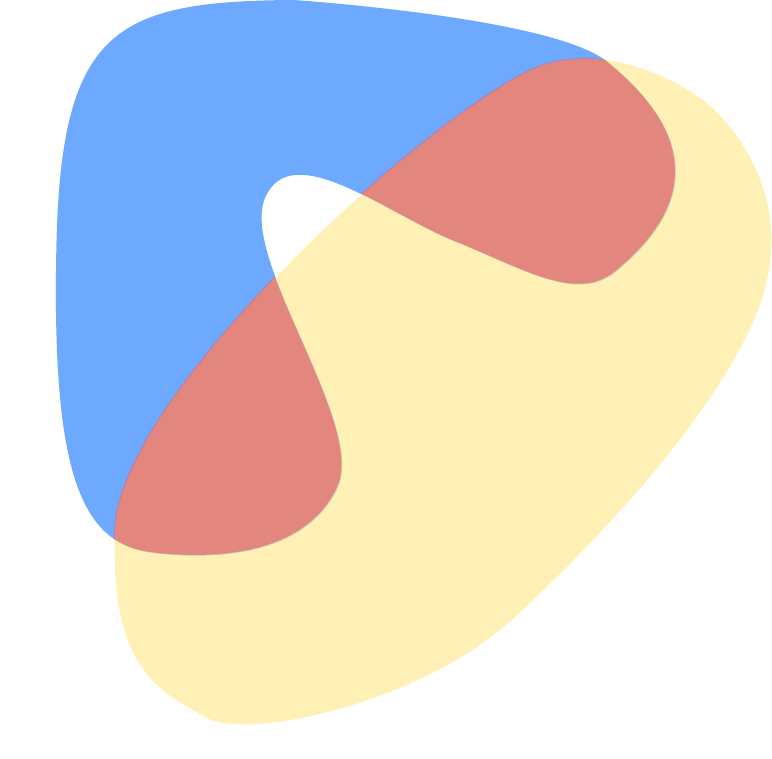}}%
    \put(0.16617811,0.83929087){\color[rgb]{0,0,0}\makebox(0,0)[lt]{\lineheight{1.25}\smash{\begin{tabular}[t]{l}$X$\end{tabular}}}}%
    \put(-0.00251162,0.00234962){\color[rgb]{0,0,0}\makebox(0,0)[lt]{\lineheight{1.25}\smash{\begin{tabular}[t]{l}$Y$\\\end{tabular}}}}%
    \put(0.83437207,0.29711841){\color[rgb]{0,0,0}\makebox(0,0)[lt]{\lineheight{1.25}\smash{\begin{tabular}[t]{l}$G^*$\end{tabular}}}}%
    \put(0,0){\includegraphics[width=\unitlength,page=2]{structure_of_cores.pdf}}%
    \put(0.08154504,0.94967995){\color[rgb]{0,0,0}\makebox(0,0)[lt]{\lineheight{1.25}\smash{\begin{tabular}[t]{l}$\H$\end{tabular}}}}%
    \put(0.19304449,0.63626579){\color[rgb]{0,0,0}\makebox(0,0)[lt]{\lineheight{1.25}\smash{\begin{tabular}[t]{l}$b$\\\end{tabular}}}}%
    \put(0.7839363,0.81933256){\color[rgb]{0,0,0}\makebox(0,0)[lt]{\lineheight{1.25}\smash{\begin{tabular}[t]{l}$c$\\\end{tabular}}}}%
    \put(0,0){\includegraphics[width=\unitlength,page=3]{structure_of_cores.pdf}}%
  \end{picture}%
\endgroup%

        \caption{Variables in the proof of \rlem{lem:manycopies}.}
        \label{fig:structure_of_cores}
    \end{figure}
We will now upper-bound the contribution of all copies $\H$ which satisfy the configuration specified by $(b, c, X, Y)$.  Observe that

    \begin{align*}
        x + y &= q,\\
        |E(Y)| + |E(X)| + b + c &= E(H) = q\D/2,\\
        2|E(X)| + b &= \D|V(X)| = x\D.
    \end{align*}
    Solving these equations we get
    \begin{align*}
        |E(X)| &= \f{x\D - b}{2},\\
        |E(Y)| &= \f{q\D}{2} - b - c - \f{x\D}{2} + \f{b}{2} = \f{q\D - x\D - b - 2c}{2}.
    \end{align*}
    Since $b + |E(X)| + c$ edges come from outside $\G$, for each such $\H$, $p_n^{|\H - \G|}$ is at most
    \begin{align*}
        C n^{-\f{2}{\D}\cdot(b + |E(X)| + c)} = C n^{-\f{2}{\D}\cdot\f{x\D + b + 2c}{2}} = C n^{-x - \f{b}{\D} - \f{2c}{\D}}.
    \end{align*}
    Now we need to count the number of such copies of $\H$. Firstly, the image of $X$ can be chosen in (at most) $n^x$ ways, which can be included in the ``cost'' to obtain $n^{-\f{b}{\D} - \f{2c}{\D}}$. The final cost will be the product of this with the number of copies of $Y$ in $\G$ (fixing it will fix all the vertices, and therefore all the edges counted in $b$ and $c$). 

    To count the number of copies of $Y$ in $\G$, observe that $Y$ being a subgraph of $H$ has max-degree $\D$. So we can apply \rlem{lem:holder} to obtain that the count is at most 
    \[ C e^{|E(Y)|/\D} v^{y - 2|E(Y)|/\D} \] 
    where $v = v(\G)$ and $e = e(\G)$. Putting everything together, the total cost is at most
    \begin{align*}
        C e^{\f{q - x}{2} - \f{b + 2c}{2\D}} v^{q - x - 2\f{q - x}{2} + \f{b + 2c}{\D}} n^{-\f{b}{\D} - \f{2c}{\D}} = C e^{\f{q - x}{2} - \f{b + 2c}{2\D}} v^{\f{b + 2c}{\D}} n^{-\f{b + 2c}{\D}} \leq C e^{\f{q}{2} - \f{x + b + 2c}{2\D}}
    \end{align*}
    as $v \leq n$ and $\D \geq 1$. If $b + 2c + x\geq 1$, then this is $\o(k_n)$ because $e = \O(k_n^{2/q})$ (and $k_n \gg \poly\log n$ as assumed at the beginning of this section). But this is indeed the case since $H$ is connected and $|\H - \G| \geq 1$,  implying that the contribution of the part of the sum corresponding to the configuration $(b, c, X, Y)$ is $\o(k_n)$. 

    Considering all possible distinct configurations of $(b, c, X, Y)$ (of which there are $C = C(H)$ many), we see that the total contribution of the second term in the first display of the proof is $\o(k_n)$. This concludes the proof.
\end{proof}

As indicated in the idea of proof section, the following ``product of degrees'' lemma is a key combinatorial input in our proof, derived via an application of Finner's inequality.

\begin{lem}[Product of degrees]
    \label{lem:product}
    There is a function $g(n) = \Om(k_n^{1/q})$, depending only on $H$, such that the following holds for all sufficiently large $n$. Suppose $\G$ is a core. For every edge $f = (a, b) \in E(\G)$, we have $d_a d_b \geq g^2(n)$ and $d_a, d_b \geq \D$. Therefore every vertex in a core has degree at least $\D$ since none of them are isolated.
\end{lem}

Every vertex which is a part of some copy of $H$ in $\G$ must have degree at least $\D$. But the lemma above does not claim that every vertex in $\G$ is a part of some copy of $H$, just the weaker statement that every vertex has degree at least $\D$. Specifically, it does not rule out the case of a vertex which has a very large degree in $\G$ but is not  part of any copy of $H$ in $\G$.
\begin{proof}
    This proof is a direct application of \rlem{lem:fixedholder}. Observe that
    \begin{align*}
         \E_\G \Q - \E_{\G - f} \Q &= \sum_{\H} p_n^{|\H - \G|} - \sum_{\H} p_n^{|\H - (\G - f)|}
        \intertext{where $|\H - \G|$ is again a shorthand for the number of edges in $\H$ not in $\G$. Therefore the above is}
                                           &= \sum_\H p_n^{|\H - \G|} - p_n^{|\H - (\G - f)|}.
                                           \intertext{Now observe that for any $\H$ which does not contain $f$, the term is zero, and for any $\H$ containing $f$, $|\H - (\G - f)| = 1 + |\H - \G|$. Therefore, the above is}
                                           &= (1 - p_n)\sum_{\H \ni f} p_n^{|\H - \G|},
                                           \intertext{where we sum over all $\H$ containing $f$. This is bounded by}
                                           &\leq \sum_{\H \ni f} p_n^{|\H - \G|}.
                                           \intertext{This last expression is the expected number of copies of $\H$ containing the edge $f$ \fix{given that $\G \subseteq G$}. We may now invoke \rlem{lem:fixedholder} to bound the above by:}
                                           &\leq C(H)(d_a + np_n^\D)^{\f{\D - 1}{\D}}(d_b + np_n^\D)^\f{\D - 1}{\D}(e + n^2p_n^\D)^{\f{q}{2} - 2 + \f{1}{\D}},
                                           \intertext{where $d_a, d_b \geq 1$ are the degrees of $a, b$ respectively in $\G$ and $e = e(\G)$. Recall that $p_n^\D = n^{-2}$, so that by modifying the constants, we can write}
                                           &\leq C(H) (d_a d_b)^\f{\D - 1}{\D} e^{\f{q}{2} - 2 + \f{1}{\D}}.
                                           \intertext{Since $\G$ is a core, $e = \O(k_n^{2/q})$, so that}
                                           &\leq (d_a d_b)^{\f{\D - 1}{\D}} \O\Rnd{k_n^{1 - \f{4}{q} + \f{2}{q\D}}}.
    \end{align*}
    Definition \ref{def:core} and the preceding chain of inequalities then imply
    \begin{align*}
        t_n = \Om\Rnd{k_n^{1 - \f{2}{q}}} \leq \E_\G \Q - \E_{\G - f} \Q \leq (d_a d_b)^{\f{\D - 1}{\D}} \O\Rnd{k_n^{1 - \f{4}{q} + \f{2}{q\D}}},
    \end{align*}
    and therefore,
    \begin{align*}
        (d_a d_b)^\f{\D - 1}{\D} \geq \Om\Rnd{k_n^{\f{2}{q} - \f{2}{q\D}}} \iff d_a d_b \geq \Om(k_n^{2/q}).
    \end{align*}
    Observe that this lower bound is uniform over all cores, and just depends on $H$. We define this function to be $g^2(n)$, completing the proof of the first claim.

    For the other part, we apply the second half of \rlem{lem:fixedholder}. Suppose $d_a < \D$. Then (given $\G \subseteq G$) there is no copy of $H$ in $\G$ containing $f$ because $H$ is $\D$-regular. Then every copy of $H$ in $G$ containing $f$ must have some other edge coming from outside $\G$. The number of such copies is bounded above in expectation, using \rlem{lem:fixedholder}, by $C(H)\max(B_1, B_2, B_3)$ where $B_1, B_2, B_3$ are as defined in the lemma. Observe that $d_b \leq e(\G) = \O(k_n^{2/q})$, so that we can bound $B_1, B_2, B_3$ as follows:
         \begin{align*}
         B_1 & \df (d_a + np^\D)^{\f{\D - 2}{\D}}(d_b + np^\D)^{\f{\D - 1}{\D}}(e + n^2 p^\D)^{q/2 - 2 + 1/\D} pn^{1/\D} \\
             & \leq \O\Rnd{\Rnd{k_n^{2/q}}^{\f{\D - 1}{\D} + \f{q}{2} - 2 + \f{1}{\D}}} n^{-1/\D} = \O(k_n^{1 - 2/q}) n^{-1/\D} = \o(k_n^{1 - 2/q}), \\
         B_2 & \df (d_a + np^\D)^{\f{\D - 1}{\D}}(d_b + np^\D)^{\f{\D - 2}{\D}}(e + n^2 p^\D)^{q/2 - 2 + 1/\D} pn^{1/\D} \\
             & \leq \O\Rnd{\Rnd{k_n^{2/q}}^{\f{\D - 2}{\D} + \f{q}{2} - 2 + \f{1}{\D}}} n^{-1/\D} = \O(k_n^{1 - 2/q}) n^{-1/\D} = \o(k_n^{1 - 2/q}), \\
         B_3 & \df (d_a + np^\D)^{\f{\D - 1}{\D}}(d_b + np^\D)^{\f{\D - 1}{\D}}(e + n^2 p^\D)^{q/2 - 2} pn^{2/\D} \\
             & \leq \O\Rnd{\Rnd{k_n^{2/q}}^{\f{\D - 1}{\D} + \f{q}{2} - 2}} = \O\Rnd{k_n^{1 - \f{2}{q} - \f{2}{q\D}}} = \o(k_n^{1 - 2/q}),
    \end{align*}
    because $k_n$ grows superpolylogarithmically. Therefore the expectation is bounded by $\o(k_n^{1 - 2/q})$. But, as in the proof of the first claim, this must be at least $t_n = \Om(k_n^{1 - 2/q})$. This is a contradiction (for sufficiently large $n$), proving the claim.
\end{proof}

The next lemma proves the theorem of this paper in the case of large $k_n$. As mentioned earlier, combined with \rlem{lem:few}, this will complete the proof of the main theorem. 

The following is a high-level overview of the main steps involved. The proof proceeds via successively specializing the class of cores whose (collective) probability we have to bound. We bi-partition the vertices of the core into high degree and low degree vertices and  using \rlem{lem:product} ensure that edges from low degree vertices can only go to sufficiently high degree vertices. Since there are a relatively small number of high degree vertices, we can fix all of them in $V(K_n) = [n]$ (without incurring too high a cost from the union bound). At this point the case where at least half of the edges in the core comes from the edges within this high-degree set is dealt with using straightforward binomial tail probabilities, leaving us with the more complicated case when at least half of the edges go between the high and low degree sets. The remainder of the proof assumes that this is the case. Now we attempt to fix all the low degree vertices as well. But the straightforward bound of $e = E(\G)$ on their count is too lossy, thus necessitating further control. However, by a dyadic decomposition, we can fix all of them except the ones with very low degree. But vertices with extremely low degree can have edges only to vertices with very high degree, again by \rlem{lem:product}, \emph{whose} number is very small allowing us to also fix all the neighbor sets of these low degree vertices (and therefore their degrees as well). However, since each vertex in the core has degree at least $\D$ (\rlem{lem:product}), the number of such low degree vertices cannot exceed $1/\D$ times the number of edges between the high and low degree sets. This improved bound by a factor of $1/\D$ indeed turns out to be sufficient to balance the cost of fixing the low degree vertices against the probability of these edges actually existing (recall that we have already fixed the neighbor sets for these vertices). A calculation then shows that this works out favorably, finishing the proof.

\begin{lem}
    \label{lem:many}
    Assume $k_n$ grows fast enough (as described at the beginning of this section). The probability that $\cG(n, p_n)$ has a core is at most $C\exp(-C' k_n^{2/q}\log n)$ for constants $C, C'$ depending on $H$. As a consequence, the probability of $Q \geq k_n$ also satisfies the same bound (up to changing $C$).
\end{lem}
\begin{proof}
    \newcommand{\m}{\mathrm{max}}
    Let $g(n) = \Om(k_n^{1/q})$ be the function in \rlem{lem:product}, and suppose $\G \subseteq K_n$ is a core. Let $L$ be the set of \emph{high-degree vertices} in $\G$ defined to be those with degree $\geq g(n)$ in $\G$, and set $R = V(\G) - L$. By \rlem{lem:product} every edge in $\G$ must have an endpoint in $L$. Also we have $|L| \leq L_\m$ where $L_\m = \O(k_n^{1/q})$ because $e(\G) = \O(k_n^{2/q})$. Note that the implicit $\poly\log n$ factor in $\O(k_n^{1/q})$ does not depend on $\G$. 

    Now recall that a core has at least $C k_n^{2/q}$ edges as was proved in \rlem{lem:manycopies}. For brevity let $e = e(\G)$. Since all these edges have an endpoint in $L$, either at least $e/2$ of them have both their endpoints in $L$ or at least $e/2$ of them have one endpoint in $L$ and the other in $R$. The probability that at least $e/2$ edges come from within $L$ is bounded by the probability that there is a set of size at most $L_\m$ containing at least $C k_n^{2/q}$ edges (with a different $C$). For a fixed set of size at most $L_\m$, the probability it has so many edges is at most
    \begin{align*}
        \P[\Bin\Rnd{\binom{L_\m}{2}, p_n} \geq C k_n^{2/q}] &\leq \P[\Bin(L_\m^2/2, p_n) \geq C k_n^{2/q}]
        \intertext{which may be upper bounded using \rlem{lem:bintails} from the Appendix (since $p_n \ll \f{Ck_n^{2/q}}{L_\m^2}$) to obtain}
                                                            &\leq \Exp{-C' k_n^{2/q} \log\Frac{C k_n^{2/q}}{L_\m^2p_n}}, \quad \text{$C, C'$ constants.}
                                                 \intertext{Recalling that $L_\m = \O(k_n^{1/q})$, we can conclude that the above is therefore at most}
                                                 &\leq \Exp{-C k_n^{2/q} \log n}, \quad 
                                                 \fix{\text{because $Ck_n^{2/q}/L_\m^2 \geq 1/\poly\log n$},}
    \end{align*}
    for sufficiently large $n$. The total number of ways of picking such a subset of size at most $L_\m$ from $V(K_n)$ is at most
    \begin{align*}
        n^{L_\m} \leq \Exp{k_n^{1/q}\pln},
    \end{align*}
    so that the probability that there is a subset of vertices of size at most $L_\m$ with at least $e/2$ edges inside it is at most
    \begin{align}
        \label{eq:lesse2}
        \Exp{k_n^{1/q}\pln - Ck_n^{2/q}\log n} \leq C'\Exp{-Ck_n^{2/q}\log n}.
    \end{align} 
    Therefore for the class of cores satisfying this condition, we have the desired bound. 

    In the sequel we will prove a similar bound for cores which do not satisfy this condition and hence have fewer than $e/2$ edges coming from within $L$ so that there are at least $e/2$ edges between $L$ and $R$. Partition $L$ into $m = O(\log n)$ pieces dyadically depending on the degree, i.e., define
    the (disjoint) sets
    \begin{align*}
        L_i \df \{v \in L : g(n) 2^{i - 1} \leq d_v < g(n)2^i \}, \quad i = 1, 2, \ldots, m 
    \end{align*}
    where $d_v$ is the degree of $v$ in $\G$. Note that in fact, we can choose $m$ such that 
    \begin{align}
        \label{eq:mdef}
        10\rho(n) < 2^m g(n) \leq 20\rho(n)
    \end{align}
    where $\rho(n)$ is the maximum number of edges in a core as in Definition \ref{def:core}, because this ensures that the sets $L_{m + 1}, \ldots$ are empty by the bound $d_v \leq e$ for all $v \in \G$ (observe that this definition does not depend on the choice of core, but only on $n$, $H$ and $k_n$).
    Also define the nested sets:
    \begin{align*}
        R_i \df \{v \in R : d_v \geq g(n) 2^{-i} \}, \quad i = 1, 2, \ldots, m
    \end{align*}
    so that $R_1 \subseteq R_2 \subseteq \ldots$ (see Figure \ref{fig:dyadic}). Invoking \rlem{lem:product}, observe that every edge $(a, b)$ between $L$ and $R$ goes from a vertex in $L_i$ to one in $R_i$, for some $i$, and that 
    \begin{align}
        \label{eq:ribnd}
        |R_i| \leq \f{2^{i + 1}e}{g(n)}.
    \end{align}
    

First we fix $m, e$ and all the $\{e_{ij}\}_{i, j \leq m}$ defined as 
    \begin{align*}
        e_{ij} \df \text{number of edges between $L_i$ and $R_j - R_{j - 1}$}, \quad R_0 \df \varnothing.
    \end{align*}
    The total number of such distinct sequences is at most $\exp(\poly\log n)$ because each of these is at most $n^2$, and there are $O(\log^2 n)$ numbers to fix. This factor will be negligible while applying a union bound, and hence it suffices to consider the class of cores which satisfy these choices. In the sequel we also use the notation
    \begin{align*}
        e_i \df \text{number of edges between $L_i$ and $R_i$} = \sum_{j = 1}^i e_{ij}.
    \end{align*}

    Now fix the sets $L_1, \ldots, L_m$. Since the total number of distinct choices of each $L_i$ is at most 
    \begin{align*}
        n^{L_\m} \leq \Exp{Ck_n^{1/q}\log n},
    \end{align*}
    the total number of ways to choose $L_1, \ldots, L_m$ is only
    \begin{align*}
        \Exp{Ck_n^{1/q}\log^2 n} = \Exp{o(k_n^{2/q})}
    \end{align*}
    so that we can fix these and bound the probability of the resulting subclass of cores, again using the fact that this factor will be negligible in a union bound.
    
    Note that by definition (see \eqref{eq:mdef})
    \begin{align*}
        2^m \leq C\frac{\rho(n)}{g(n)} \leq \f{\O(k_n^{2/q})}{\Om(k_n^{1/q})} = \O(k_n^{1/q}).
    \end{align*}
    Then there is a function $h(n) = o(k_n^{2/q})$ depending only on $H$ such that we can choose $s = \Theta(\log \log n)$ satisfying
    \begin{align}
        \sum_{i = 1}^{m - s} |R_i| \leq \sum_{i = 1}^{m - s} 2^{i + 1} \f{e}{g(n)} \leq 2^{m - s + 2} \f{e}{g(n)} = \f{\O(k_n^{1/q})\O(k_n^{1/q})}{2^s}\leq h(n),
        \label{eq:rbound}
    \end{align}
    and also for all $i > m - s$,
    \begin{align}
        \label{eq:libnd}
        |L_i| \leq \f{C e}{2^i g(n)} \leq C' \frac{e}{g(n)}\cdot\frac{g(n)}{\rho(n)} \cdot 2^s \leq C' \cdot 2^s \leq \poly\log n, \quad \text{$C, C'$ constants.}
    \end{align}
    Observe that $s$ only depends on $n$ and $H$, and is fixed beforehand. 
    For this choice of $s$, we can similarly fix $R_1, \ldots, R_{m - s}$ by first fixing the sizes of all the $R_i$ (which is a sequence of length $O(\log n)$ each of value at most $n$, so that total number of ways is $\exp(\pln)$) satisfying the constraint that
    \begin{align*}
        \sum_{i = 1}^{m - s} |R_i| \leq h(n), \quad \text{(same as in \eqref{eq:rbound})}
    \end{align*}
    and then choosing these sets as subsets of $[n]$ in 
    \begin{align*}
        \Exp{\sum_{i = 1}^{m - s} |R_i| \log n} = \Exp{o(k_n^{2/q} \log n)}\ \text{ways.} 
    \end{align*}
   
    \begin{figure}[H]
        \centering
    \def\svgwidth{0.4\columnwidth}
    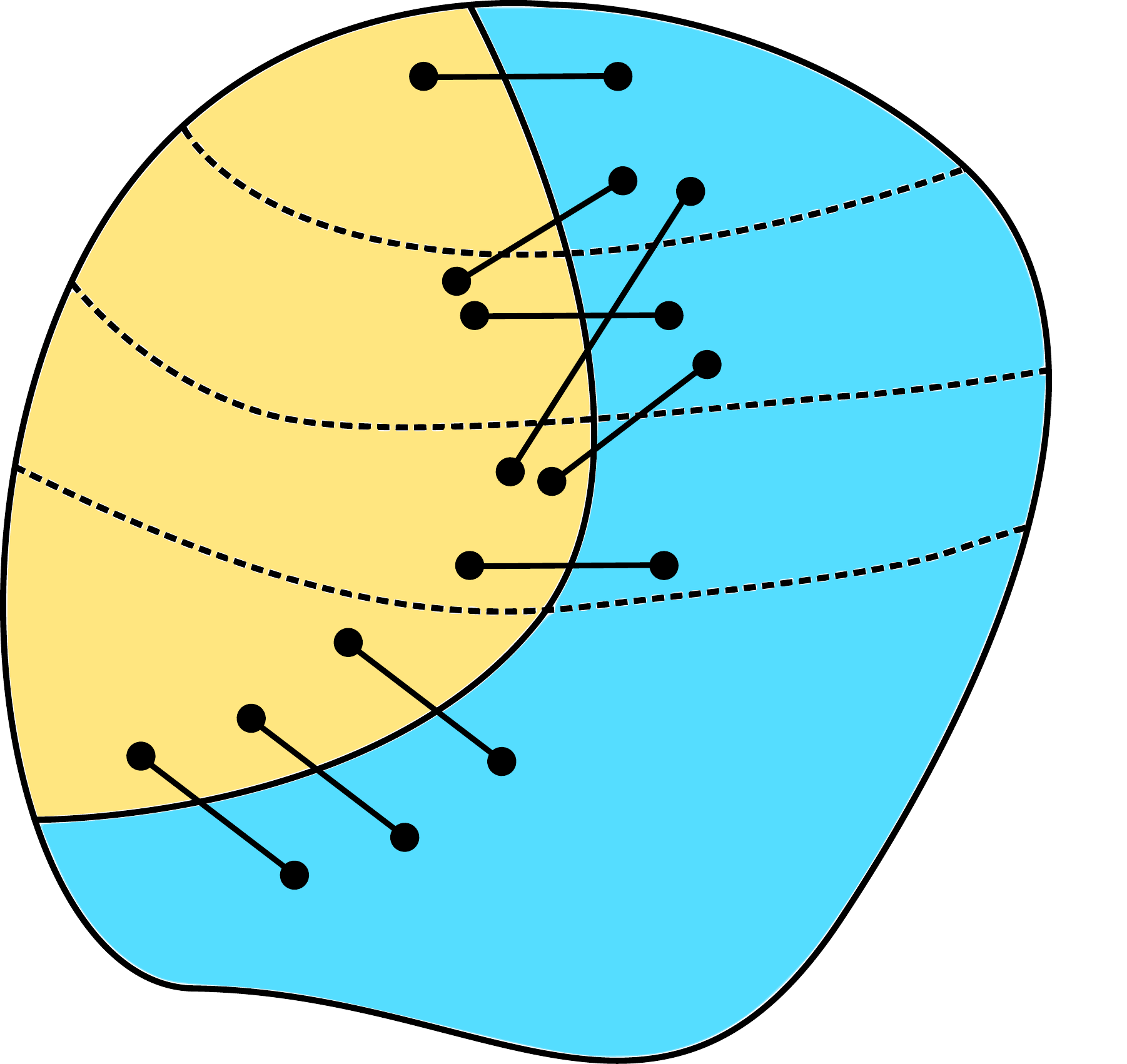

        \caption{Variables in the proof of \rlem{lem:many} where $m - s = 3$.}
        \label{fig:dyadic}
    \end{figure}

    Denote by $L' \df \cup_{i > m - s} L_i$, $R' \df \fix{R - R_{m - s}}$, and $e'$ for the number of edges between $L'$ and $R'$ (see Figure \ref{fig:dyadic}). Observe that
    \begin{align*}
        e' = \sum_{m - s < i, j \leq m} e_{ij}.
    \end{align*}
    \fix{
        Since $|L'| \leq \poly\log n$ and $|R_{m - s}| \leq h(n)$, we may exclude the case of at least $e/4$ edges between $L'$ and $R_{m - s}$ by bounding the probability that there is any pair of sets $A, B \subseteq [n]$ of size at most $\poly\log n$ and $h(n) \leq o(k_n^{2/q})$ (respectively) such that they have at least $e/4 \geq Ck_n^{2/q}$ edges between them. We may choose the sizes of $A$ and $B$ and then $A$ and $B$ themselves in at most 
        \begin{align*}
            \Exp{\pln} \binom{n}{\poly\log n}\binom{n}{h(n)} \leq \Exp{o(k_n^{2/q})\log n}
        \end{align*}
        ways, and after fixing $A$ and $B$, the probability of there being at least $C k_n^{2/q}$ edges between them is at most
        \begin{align*}
            \P(\Bin(\poly\log n \cdot h(n), p_n) \geq C k_n^{2/q}) & \leq \Exp{-C k_n^{2/q}\log\f{C k_n^{2/q}}{\poly\log n \cdot h(n)\cdot n^{-2/\D}}} \\
                                                                   &\leq \Exp{-C k_n^{2/q} \log n}
        \end{align*}
        using \rlem{lem:bintails} and the bound $h(n) \leq o(k_n^{2/q})$.
        Therefore, by a union bound over all $A, B$, this probability is at most of the same order as our eventual bound, so we may restrict ourselves to the case when the number of edges between $L'$ and $R_{m - s}$ is at most $e/4$. Therefore, since 
        \begin{align*}
            \sum_{i = 1}^{m - s} e_i + e' + \textrm{(number of edges between $L'$ and $R_{m - s}$)} = (\textrm{number of edges between $L$ and $R$}) > e/2,
        \end{align*}
        we can restrict ourselves to the case when $\sum_{i = 1}^{m - s} e_i + e' > e/4$. We assume this in the sequel.
    }

    Note that due to our choices, $L'$ and $e'$ are fixed (since all the $L_i$ and $e_{ij}$ were fixed earlier), but $R'$ depends on the core. However every edge between $L$ and $R$ with an endpoint in $R'$ must go to $L'$. Fix the size $r' = |R'|$ (in $n$ ways) and then fix the degrees of the vertices in $R'$, say $d_1, \ldots, d_{r'}$. Since we must ensure $d_1 + \ldots + d_{r'} = e'$, and each degree $d_i \leq |L'| \leq \pln$, the number of ways to choose these is at most
    \begin{align*}
        (\pln)^{e'} \leq \exp(ce' \log \log n),
    \end{align*}
    where $c$ is an absolute constant. Being in a core, vertices in $R'$ must have degree at least $\D$ (due to \rlem{lem:product}). Now we also fix the neighbor sets $S_i$ (with $|S_i| = d_i$) of each these vertices, the number of ways of which is at most
    \begin{align*}
        |L'|^{\sum_{i = 1}^{r'} d_i} \leq \Exp{e' \log |L'|} \leq \Exp{c e' \log\log n}.
    \end{align*}
    Therefore, we can fix $r', d_1, \ldots, d_{r'}, S_1, \ldots, S_{r'}$ in a total of 
    \begin{align*}
        n \Exp{c e' \log \log n} 
    \end{align*}
    ways. We will prove a final probability bound of $\exp(-ce \log n)$ on this class of cores, so we can ignore these factors (observe that this bound is sufficient because $e \geq Ck_n^{2/q}$ as claimed earlier).

    For the final step of the proof, we first summarize the information we have now. For each distinct class of cores, we know
    \begin{itemize}
        \item $m = O(\log n)$, the number of dyadic pieces,
        \item $e$, the number of edges in the core,
        \item $L_1, \ldots, L_m$, all the pieces in the high-degree set,
        \item $R_1, \ldots, R_{m - s}$, all except the last $s$ pieces in the low-degree set,
        \item $e_1, \ldots, e_m$, the number of edges between every $L_i$ and $R_i$ (in fact we know all the $e_{ij}$),
        \item $\sum_{i = 1}^{m - s} e_i + e' \geq e/4$. Define $e'' = \sum_{i = 1}^{m - s} e_i$.
        \item $r'$, the size of $R'$.
        \item $d_1, \ldots, d_{r'}$, the degrees of vertices in $R'$.
        \item and $S_1, \ldots, S_{r'}$, their associated neighbor sets (which are known to be subsets of $L'$).
    \end{itemize}
    with a total count of $\Exp{o(k_n^{2/q}\log n)}$ many such classes. We now wish to show a total probability bound of $\Exp{-Ck_n^{2/q}\log n}$ on each such class of cores. Note that this is sufficient to complete the proof. For the rest of this argument, we fix such a class and show this claimed upper bound.

    Choose $R'$ and a bijection $\psi: [r'] \to R'$ in at most $\exp(r' \log n)$ ways. But each degree is at least $\D$, therefore $e' \geq r'\D \implies r' \leq e'/\D$, so that the total number of ways is at most $\exp(\f{e'}{\D}\log n)$. Fixing $R'$ and $\psi$ also fixes all the edges between $R'$ and $L'$ via the information about $S_1, \ldots, S_{r'}$ so that the probability of this class is then at most
\begin{align*}
    \Exp{-e' \log (1/p_n)} \prod_{i = 1}^{m - s} \P(\Bin(|L_i||R_i|, p_n) \geq e_i),
\end{align*}
where the second factor accounts for $e_i$ edges between $L_i$ and $R_i$.
Multiplying the number of ways to fix $R'$ and $\psi$ we get
\begin{align*}
    \Exp{\f{e'}{\D} \log n - \f{2}{\D}e'\log n} \prod_{i = 1}^{m - s}\P(\Bin(|L_i||R_i|, p_n) \geq e_i),
\end{align*}
where we substitute $p_n = n^{-2/\D}$. If $e' \geq e/8$, we are done, because the first factor itself is at most $\exp(-ce \log n) \leq \Exp{-Ck_n^{2/q}\log n}$. If not, we must have $e''\df \sum_{i = 1}^{m - s} e_i \geq e/8$. In that case the first term is at most 1 and the second term is at most
\begin{align*}
    \P[\Bin\Rnd{\sum_{i = 1}^{m - s}|L_i||R_i|, p_n} \geq e/8] & \leq \P[\Bin\Rnd{\O(k_n^{2/q}), p_n} \geq e/8]
    \intertext{because from \eqref{eq:ribnd} and \eqref{eq:libnd} we have $|L_i||R_i| \leq C\f{e^2}{g(n)^2} = \O(k_n^{2/q})$ and $m = O(\log n)$. Applying \rlem{lem:bintails} we bound the above by}
                                                               & \leq \Exp{-C e \log\Frac{e}{8\O(k_n^{2/q})n^{-2/\D}}} 
                                                               \intertext{using the fact that $p_n = n^{-2/\D} \ll e/\O(k_n^{2/q}) = \Om(1)$. So the quantity inside the $\log$ is $\Om(n^{2/\D})$ yielding the bound,}
                                                               & \leq \Exp{-C e\log n}
\end{align*}
for a constant $C$, as required.
\end{proof}

Combining the results of \rlem{lem:few} and \rlem{lem:many} finishes the proof of our main theorem, Theorem \ref{thm:main}.


\section{Appendix}

This appendix collects useful lemmas that were required in the proofs above.
\begin{lem}
    \label{lem:concave}
    Let $p > 1$ be a real. Then for any sequence of nonnegative reals $(x_i)_{i = 1}^n$ we have $\sum_i x_i^{1/p} \geq \Rnd{\sum_i x_i}^{1/p}$.
\end{lem}
\begin{proof}
    Let $s = \sum x_i$, and $y_i = x_i/s$ for $i \in [n]$. Then $y_i \leq 1$, so $y_i^{1/p} \geq y_i$. Therefore $\sum_i y_i^{1/p} \geq \sum_i y_i = 1$. Substituting back the values of $y_i$, we obtain the result.
\end{proof}

\begin{lem}
    \label{lem:competition}
    For any $k \geq 2$ and $A > 0$,
    \begin{align*}
        \min_{0 \leq s \leq k} \Rnd{s\log s + A(k - s)^{2/q}} \geq c \min(k \log k, A k^{2/q}),
    \end{align*}
    where $c$ is an absolute constant.
\end{lem}
\begin{proof}
    Let $s_*$ be the optimal value of $s$. If $s_* \geq k/2$, then the above quantity is at least
    \begin{align*}
        \f{1}{2}k \log (k/2) \geq \f{1}{10} k \log k,
    \end{align*}
    as long as $k \geq 3$. If not, $k - s_* \geq k/2$ in which case the result is clear. Finally if $k = 2$, the result can be verified manually.
\end{proof}
Next we prove an estimate about the tails of the binomial distribution. A related result already appeared in \cite[Lemma 3.3]{lubetzkyzhaovariational}.
\begin{lem}[Binomial tails]
    \label{lem:bintails}
Suppose $N = N(n), M = M(n) \to \infty$ and $p_n = o(M/N)$. Then
    \begin{align*}
        \P(\Bin(N, p_n) \geq M) \leq \Exp{-(1 - o(1)) M \log\Frac{M}{Np_n}}.
    \end{align*}
\end{lem}
\begin{proof}
    By Chernoff's method we know that,
    \begin{align*}
        \P(\Bin(N, p_n) \geq M) \leq \Exp{-N H_{p_n}(M/N)},
    \end{align*}
    where $H_p(t) = D_{\text{KL}}(\Ber(t)\ \|\ \Ber(p))$,
    so it suffices to show that $H_{p_n}\Frac{M}{N} \geq (1 - o(1))\f{M}{N}\log\Frac{M}{Np_n}$, which can be reduced to proving that for $p \ll x < 1$, we have 
    \begin{align*}
        H_p(x) \geq (1 - o(1))x \log (x/p).
    \end{align*}
    It is indeed well-known that $H_p(x) \sim x \log (x/p)$ in this regime, see for instance \cite[Lemma 3.3]{lubetzkyzhaovariational}. For the purposes of completeness we include a proof of this weaker variant.
    To see this recall that
    \begin{align*}
        H_p(x) = x\log\Frac{x}{p} + \underbrace{(1 - x)\log\Frac{1 - x}{1 - p}}_{< 0},
    \end{align*}
    so it suffices to show that 
    \begin{align*}
        (1 - x)\log\f{1 - p}{1 - x} = o(1) \cdot x \log\f{x}{p} \iff \f{1 - x}{x} \f{\log\f{1 - p}{1 - x}}{\log\f{x}{p}} \to 0 \iff \f{1 - x}{x}\f{\log\Rnd{1 + \f{x - p}{1 - x}}}{\log\f{x}{p}} \to 0.
    \end{align*}
    Now $\log(1 + x) \leq x$, so this is implied by
    \begin{align*}
        \f{1 - x}{x} \f{x - p}{1 - x} \log^{-1}(x/p) \to 0
    \end{align*}
    which is clear because $x - p \leq x$ and $x/p \to \infty$.
\end{proof}



\renewcommand*{\bibfont}{\footnotesize}
\printbibliography 



\end{document}